\documentclass[12pt,oneside,english]{amsart}

\usepackage[T1]{fontenc}
\usepackage[latin9]{inputenc}
\usepackage{geometry}
\usepackage{babel}
\usepackage{amsthm}
\usepackage[unicode=true,pdfusetitle,
 bookmarks=true,bookmarksnumbered=false,bookmarksopen=false,
 breaklinks=false,pdfborder={0 0 1},backref=false,colorlinks=false]
 {hyperref}
\usepackage{breakurl}

\makeatletter
\numberwithin{equation}{section}
\numberwithin{figure}{section}

\usepackage{algorithm}
\usepackage{algorithmic}
\usepackage[all,cmtip]{xy}
\usepackage{cite}
\usepackage{tikz}
\usetikzlibrary{shapes,arrows}
 \newtheorem{thm}{Theorem}[section]
 \newtheorem{cor}[thm]{Corollary}
 \newtheorem{lem}[thm]{Lemma}
 
 \theoremstyle{definition}
 
 \theoremstyle{remark}
 \newtheorem{rem}[thm]{Remark}
 
 \numberwithin{equation}{section}

\begin{document}

%-------------------------------------------------------------------------
% editorial commands: to be inserted by the editorial office
%
%\firstpage{1} \volume{228} \Copyrightyear{2004} \DOI{003-0001}
%\seriesextraline{This is the Concrete Title of this Book\br H.E. R and S.T.C. W, Eds.}
%\firstpage{1}

\title[Universal Controllers and System Identification]
 {Universal Algebraic Controllers and System Identification}

%----------Author 1
\author[Fredy Vides]{Fredy Vides}

\address{%
Scientific Computing Innovation Center\\
School of Mathematics and Computer Science\\
Universidad Nacional Aut\'onoma de Honduras\\
Tegucigalpa\\
Honduras}

\email{fredy.vides@unah.edu.hn}

%\thanks{This work was completed with the support of our
%----------Author 2
%\author{A Second Author}
%\address{The address of\br
%the second author\br
%sitting somewhere\br
%in the world}
%\email{dont@know.who.knows}


\date{\today}

\begin{abstract}
In this document, some structured operator approximation theoretical methods for system identification of nearly eventually periodic systems, are presented. Let $\mathbb{C}^{n\times m}$ denote the algebra of $n\times m$ complex matrices. Given $\varepsilon>0$, an arbitrary discrete-time dynamical system $(\Sigma,\mathcal{T})$ with state-space $\Sigma$ contained in the finite dimensional Hilbert space $\mathbb{C}^n$, whose state-transition map $\mathcal{T}:\Sigma\times ([0,\infty)\cap \mathbb{Z})\to \Sigma$ is unknown or partially known, and needs to be determined based on some sampled data in a finite set $\hat{\Sigma}=\{x_t\}_{1\leq t\leq m}\subset \Sigma$ according to the rule $\mathcal{T}(x_t,1)=x_{t+1}$ for each $1\leq t\leq m-1$, and given $x\in \hat{\Sigma}$. We study the solvability of the existence problems for two triples $(p,A,\varphi)$ and $(p,A_\eta,\Phi)$ determined by a polynomial $p\in \mathbb{C}[z]$ with $\deg(p)\leq m$, a matrix root $A\in\mathbb{C}^{m\times m}$ and an approximate matrix root $A_\eta\in\mathbb{C}^{r\times r}$ of $p(z)=0$ with $r\leq m$, two completely positive linear multiplicative maps $\varphi:\mathbb{C}^{m\times m}\to \mathbb{C}^{n\times n}$ and $\Phi:\mathbb{C}^{r\times r}\to \mathbb{C}^{n\times n}$, such that $\|\mathcal{T}(x,t)-\varphi(A^t)x\|\leq\varepsilon$ and $\|\Phi(A_\eta^t)x-\varphi(A^t)x\|\leq\varepsilon$, for each integer $t\geq 1$ such that $\|\mathcal{T}(x,t)-y\|\leq \varepsilon$ for some $y\in \hat{\Sigma}$. Some numerical implementations of these techniques for the reduced-order predictive simulation of dynamical systems in continuum and quantum mechanics, are outlined.
\end{abstract}

%%% ----------------------------------------------------------------------
\maketitle
%%% ----------------------------------------------------------------------
%\tableofcontents
\section{Introduction}

In this document, we study some structured operator approximation problems that arise in the fields of system identification and model  order reduction of large-scale dynamical systems. 

The purpose of this document is to present some novel theoretical and computational techniques for constrained approximation and identification of data-driven systems in continuum and quantum mechanics. These systems can be interpreted as discrete-time systems that can be {\em partially} described by difference equations of the form
\begin{equation}
\Sigma: \left\{
\begin{array}{l}
x_{t+1}=S(x_t,t),~~ t\in \mathbb{Z}\cap [0,\infty)\\
x_1\in \Sigma\subseteq \mathbb{C}^n
\end{array}
\right.
\label{data_driven_sys_def}
\end{equation}
where $\Sigma\subseteq\mathbb{C}^n$ is the set of {\em valid} states for the system, and where $S:\mathbb{C}^n\times \mathbb{Z}\cap [0,\infty)\to \mathbb{C}^n$ is some constrained map that is either partially known, or needs to be determined/discovered based on some (sampled) data $\{x_t\}_{1\leq t\leq N}$, obtained in the form of data {\em snapshots} related to the system $\Sigma$ under study. One can also interpret the map $S$ in \eqref{data_driven_sys_def} using the transition block diagram \eqref{data_driven_block_diagram} as a {\em black-box device} $\mathbf{\mathfrak{S}}$, that needs to be determined in such a way that it can be used to transform the {\em present state} $x_t$ into the {\em next state} $x_{t+1}$, according to \eqref{data_driven_sys_def}.

\begin{equation}
\tikzstyle{block} = [draw, fill=white, rectangle, 
    minimum height=3em, minimum width=6em]
\tikzstyle{sum} = [draw, fill=white, circle, node distance=1cm]
\tikzstyle{input} = [coordinate]
\tikzstyle{output} = [coordinate]
\tikzstyle{pinstyle} = [pin edge={to-,thin,black}]
\begin{tikzpicture}[auto, node distance=2cm,>=latex']

    \node [input, name=input] {};
    \node [block, right of=input, node distance=2cm] (K) {$\mathbf{\mathfrak{S}}$};

    \draw [->] (input) -- node[name=u] {$x_t$} (K);    
    
    \node [output, right of=K] (output) {};
    \draw [->] (K) -- node [name=y] {$x_{t+1}$}(output); 
    \node [near end] {$ $} (input); 
\end{tikzpicture}
\label{data_driven_block_diagram}
\end{equation}
A state-transition map $\mathcal{T}$ for a data-driven system determined by \eqref{data_driven_sys_def} is a map $\mathcal{T}:\Sigma\times \mathbb{Z}\cap [0,\infty)\to \Sigma$ that satisfies the recurrence relation $x_{t+s}=\mathcal{T}(x_t,s)$ for each $x_t\in \Sigma$ that satisfies \eqref{data_driven_sys_def}.

The discovery, simulation and predictive control of the evolution laws for systems of the form \eqref{data_driven_sys_def} are highly important in predictive data analytics, for models related to the automatic control of systems and processes in industry and engineering in the sense of \cite{BROCKETT20081,finite_quantum_control_systems}. The motivation for the development of the techniques presented in this paper came from matrix approximation problems that arise in the fields of system identification and predictive data analytics in the sense of \cite{finite_state_machine_approximation, DMD_Kutz}. We approach the solution to these problems by developing some data-driven operator-theoretic methods that combine ideas and results from C$^\ast$-representation theory and multivariate statistical analysis, in order to compute approximate matrix representations of state-transition maps $\mathcal{T}$ that one can use for system identification and model order reduction. 

For the study reported in this document, we build on the abstract machinery introduced by Brockett and Willsky in \cite{finite_state_systems} and by Farhood and Dullerud in \cite{FARHOOD2002417}. Our contribution is the application of some of the operator theoretic techniques developed in \cite{Algebraic_sets} to extend the results in \cite{finite_state_systems} to data-driven systems, we also develope new theoretical and computational procedures for the eventually periodic approximation of data-driven systems, using structured perturbations of covariance matrices, completely positive linear multiplicative maps, and discrete-time systems whose time evolution is approximately controlled by algebraic matrix sets.

In this document, special attention will be given to nearly eventually periodic systems, that is, systems that in some suitable sense that will be detailed in \S\ref{section_UAC}, approximately behave as an eventually periodic system in the sense of \cite{FARHOOD2002417}.

Although, on this paper we will focus on the solution of the theoretical problems related to the existence and computability of finite-state approximation of data-driven systems determined by data sequences described by \eqref{data_driven_sys_def}, the constructive nature of the results presented in this document allows one to derive prototypical algorithms like the ones presented in \S\ref{section_algorithms}. Some numerical implementations of this prototypical algorithms are presented in \S\ref{section_experiments}.

\section{Preliminaries and Notation}
\label{notation}

Since in this study, the information about a given system is provided essentially by orbits ({\em data sequences}) in some valid state space $\Sigma$, from here on, we will refer to data-driven systems of the form \eqref{data_driven_sys_def} in terms of sets or elements in a state space $\Sigma$. Given a state-transition map $\mathcal{T}:\Sigma\times \mathbb{Z}\cap [0,\infty)\to \Sigma$ corresponding to a data-driven system $\Sigma$ of the form \eqref{data_driven_sys_def} and given a state $x\in \Sigma$, we will write $\mathcal{T}_s(x)$ to denote the state determined by the operation $\mathcal{T}(x,s)$.

We will write $\mathbb{Z}^+$ to denote the set of positive integers $\mathbb{Z}\cap [1,\infty)$. In this document the symbol $\mathbb{C}^{n\times m}$ will denote the algebra of $n\times m$ complex matrices, and we will write $\mathbf{1}_n$ to denote the identity matrix in $\mathbb{C}^{n\times n}$ and $\mathbf{0}_{n,m}$ to denote the zero matrix in $\mathbb{C}^{n\times m}$. From here on, given a matrix $X\in \mathbb{C}^{m\times n}$, we will write $X^\ast$ to denote the conjugate transpose of $X$ determined by $X^\ast=\overline{X^\top}=[\overline{X}_{ji}]$ in $\mathbb{C}^{n\times m}$. We will represent vectors in $\mathbb{C}^n$ as column matrices in $\mathbb{C}^{n\times 1}$.

Given $x\in\mathbb{C}^n$ we will write $\|x\|$ to denote the norm induced by the usual inner product in the Hilbert space $\mathbb{C}^n$ determined by $\|x\|=\sqrt{x^\ast x}=(\sum_{j=1}^n |x_j|^2)^{1/2}$. Given $A\in\mathbb{C}^{n\times n}$ we will write $\|A\|$ to denote the spectral (operator) norm in $\mathbb{C}^{n\times n}$ determined by $\|A\|=\sup_{\|x\|=1}\|Ax\|$.

Given a matrix $Z\in \mathbb{C}^{n\times n}$, and a polynomial $p\in \mathbb{C}[z]$ over the complex numbers determined by the expression $p(z)=a_0+a_1z+\cdots+a_mz^m$, we will write $p(Z)$ to denote the matrix in $\mathbb{C}^{n\times n}$ defined by the expression $p(Z)=a_0\mathbf{1}_n+a_1Z+\cdots+a_mZ^m$. We will write $\mathcal{Z}_m(p)$ to denote the algebraic set of matrix roots of $p(x)=0$ determined by the expression $\{X\in\mathbb{C}^{m\times m}~|~p(X)=\mathbf{0}_{n,n}\}$, and given $\varepsilon>0$ we will write $\mathcal{Z}_{m,\varepsilon}(p)$ to denote the set of $\varepsilon$-approximate matrix roots of $p(x)=0$ determined by the expression $\{X\in\mathbb{C}^{m\times m}~|~ \|p(X)\|\leq \varepsilon\}$

Given $\varepsilon>0$ and $A\in\mathbb{C}^{n\times n}$, we will write $\sigma_\varepsilon(A)$ to denote the $\varepsilon$-pseudospectrum of $A$, that by \cite[Theorem 2.1]{bookPspectra} is equivalent to the set of $z\in \mathbb{C}$ such that
\begin{equation}
\|(z\mathbf{1}_n-A)v\|<\varepsilon
\label{pspectra_def}
\end{equation}
for some $v\in\mathbb{C}^n$ with $\|v\|=1$. 

In this document we will write $\hat{e}_{j,n}$ to denote the matrices in $\mathbb{C}^{n\times 1}$ representing the canonical basis of $\mathbb{C}^{n}$ (the $j$-column of the $n\times n$ identity matrix), that are determined by the expression
\begin{equation}
\hat{e}_{j,n}=
\begin{bmatrix}
\delta_{1,j} & \delta_{2,j} & \cdots & \delta_{n-1,j}  & \delta_{n,j}
\end{bmatrix}^\top
\label{ej_def}
\end{equation}
for each $1\leq j\leq n$, where $\delta_{k,j}$ is the Kronecker delta determined by the expression.
\begin{equation}
\delta_{k,j}=
\left\{
\begin{array}{l}
1, \:\: k=j\\
0, \:\: k\neq j
\end{array}
\right.
\label{delta_def}
\end{equation}

We will write $\mathbf{e}_n$ to denote the vector in $\mathbb{C}^n$ with all of its components equal to 1. Given a vector $x\in \mathbb{C}^n$ we will write $\mu(x)$ and $\hat{\sigma}(x)$ to denote the mean and standard deviation of $x$, respectively, with $\mu(x)$ and $\hat{\sigma}(x)$ defined by the following expressions.
\begin{align}
\mu(x)&=\frac{1}{n}\mathbf{e}_n^\ast x\\
\hat{\sigma}(x)&=\frac{1}{\sqrt{n-1}}\|x-\mu(x)\mathbf{e}_n\|
\end{align}

Given a data matrix $X\in \mathbb{C}^{N\times r}$ with $N\geq 2$, we will write $\mathrm{cov}(X)$ to denote the covariance matrix of $X$ in $\mathbb{C}^{r\times r}$ defined by the expression.
\begin{equation}
\mathrm{cov}(X)=\frac{1}{N-1}\left(X-\frac{1}{N}\mathbf{e}_N\mathbf{e}_N^\ast X\right)^\ast\left(X-\frac{1}{N}\mathbf{e}_N\mathbf{e}_N^\ast X\right)
\label{cov_matrix}
\end{equation}

We will write that a matrix $A\in \mathbb{C}^{m\times n}$ is a $(0,1)$-matrix if $A$ is an integer matrix in which each element is a $0$ or $1$.

\section{Universal Algebraic Controllers}

\label{section_UAC}

We will say that an orbit $\{x_t\}_{t\geq 1}\subseteq \Sigma\subseteq \mathbb{C}^{n}$ of a data-driven system $\Sigma$ is {\em nearly eventually periodic} ({\bf NEP}), if for any $\varepsilon>0$ there are two integers $T'\geq 1$, $\tau'\geq 0$, and a vector sequence $\{\hat{x}_t\}_{t\geq 1}\subset \mathbb{C}^{n}$ of vectors such that.
\begin{equation}
\left\{
\begin{array}{l}
\hat{x}_1\neq 0\\
\|\hat{x}_t-x_t\|\leq \varepsilon\\
\hat{x}_{t+\tau'+T'}=\hat{x}_{t+\tau'}
\end{array}
\right., t\in\mathbb{Z}^+
\label{index_def}
\end{equation}
For some fixed $\varepsilon>0$. Let us consider the smallest integers $0\leq \tau\leq \tau'$ and $1\leq T\leq T'$, for which the relations \eqref{index_def} hold, the pair $(\tau,T)$ will be called the $\varepsilon$-index of the orbit $\{x_t\}_{t\geq 1}$, and will be denoted by $\mathrm{ind}_\varepsilon(\{x_t\})$.

We approach the solution to the system identification and predictive simulation problems for a data-driven system $\Sigma$, combining some ideas and results from C$^\ast$-representation theory and multivariate statistical analysis, in order to compute approximate matrix representations of state-transition maps $\mathcal{T}$ that need to be determined based on some sampled-data in $\Sigma$, that one can use for system identification and model order reduction. 

More specifically, given $\varepsilon>0$, an arbitrary discrete-time dynamical system $(\Sigma,\mathcal{T})$ with state-space $\Sigma$ contained in the finite dimensional Hilbert space $\mathbb{C}^n$, whose state-transition map $\mathcal{T}:\Sigma\times ([0,\infty)\cap \mathbb{Z})\to \Sigma$ is unknown or partially known, and needs to be determined based on some sampled data in a finite set $\hat{\Sigma}=\{x_t\}_{1\leq t\leq m}\subset \Sigma$ according to the rule $\mathcal{T}(x_t,1)=x_{t+1}$ for each $1\leq t\leq m-1$, and given $x\in \hat{\Sigma}$. We study the solvability of the existence problem for the triples $(p,A,\varphi)$ and $(p,A_\eta,\Phi)$ determined by a polynomial $p\in \mathbb{C}[z]$ with $\deg(p)\leq m$, a matrix $A$ in the algebraic set $\mathcal{Z}_{m}(p)\subset \mathbb{C}^{m\times m}$, a matrix $A_\eta$ in the set $\mathcal{Z}_{r,\varepsilon}(p)\subset \mathbb{C}^{r\times r}$ with $r\leq m$, a linear map $\varphi:\mathbb{C}^{m\times m}\to \mathbb{C}^{n\times n}$, and a completely positive linear multiplicative map $\Phi:\mathcal{Z}_{r,\varepsilon}(p)\to \mathbb{C}^{n\times n}$, such that the following constraints are satisfied
\begin{align}
&\|\mathcal{T}_t(x)-\varphi(A^t)x\|\leq\varepsilon \label{universal_algebraic_constraints_2}\\
&\|\Phi(A_\eta^t)x-\varphi(A^t)x\|\leq\varepsilon \label{universal_algebraic_constraints_3}
\end{align}
for each integer $t\geq 1$ such that $\|\mathcal{T}_t(x)-y\|\leq \varepsilon$ for some $y\in \hat{\Sigma}$. The triples $(p,A,\varphi)$ and $(p,A_\eta,\Phi)$ determined by the sampled data $\hat{\Sigma}$, the {\em "initial point"} $x\in \hat{\Sigma}$, and error tolerance $\varepsilon>0$, will be called the {\em \bf cyclic} and {\em \bf reduced order cyclic} {\em universal algebraic controllers} {\bf UAC} for the system $\Sigma$ with respect to the triple $(\hat{\Sigma},x,\varepsilon)$, respectively, and this relations will be represented using the expressions $(p,A,\varphi)=\mathfrak{U_{C}}(\hat{\Sigma},x,\varepsilon)$ and $(p,A_\eta,\Phi)=\mathfrak{U_{RC}}(\hat{\Sigma},x,\varepsilon)$.

\begin{rem}
\label{rem_1}
Given some sampled data $\hat{\Sigma}=\{x_1,\tilde{x}_2,\ldots,\tilde{x}_N\}$ in an orbit $\{x_t\}_{t\geq 1}$ of a NEP data-driven system $\Sigma$ determined by \eqref{data_driven_sys_def}, without loss of generality we can decompose the computation of the UAC $(p,A,\varphi)=\mathfrak{U_{C}}(\hat{\Sigma},x,\varepsilon)$ and $(p,A_\eta,\Phi)=\mathfrak{U_{RC}}(\hat{\Sigma},x,\varepsilon)$ as follows.
\begin{enumerate}
\item[\ref{rem_1}.1] Estimate $(s,T)=\mathrm{ind}_\varepsilon(\{x_t\})$ based on $\hat{\Sigma}=\{x_t\}_{1\leq t\leq N}$ and set $p(z)=z^{s+T}-z^s$\;
\item[\ref{rem_1}.2] Compute $A\in \mathcal{Z}_m(p)$ and a completely positive linear multiplicative map $\varphi:\mathbb{C}^{m\times m}\to \mathbb{C}^{n\times n}$ that satisfy \eqref{universal_algebraic_constraints_2}\;
\item[\ref{rem_1}.3] Compute $A_\eta\in\mathcal{Z}_{r,\varepsilon}(p)$ with $r\leq m$ and a completely positive linear multiplicative map $\Phi:\mathbb{C}^{r\times r}\to \mathbb{C}^{n\times n}$ that satisfy \eqref{universal_algebraic_constraints_3}\;
\end{enumerate}
In this document, the maps $\varphi:\mathbb{C}^{m\times m}\to \mathbb{C}^{n\times n}$ and $\Phi:\mathbb{C}^{m\times m}\to \mathbb{C}^{n\times n}$ determined by the UAC of the system $\Sigma$, will be called the {\em cyclic realization} ({\bf CR}) and the {\em reduced cyclic realization} ({\bf RCR}) of the system $\Sigma$, respectively.
\end{rem}

We will now study the three steps in remark \ref{rem_1} that are involved in the computation of the unversal algebraic controllers of a data-driven system $\Sigma$.

\subsection{$\varepsilon$-Indices and nearly eventually periodic orbits}
\label{Indices}
Given $\varepsilon>0$ and some sampled-data $\{\tilde{x}_t\}_{t=1}^N$ from the orbit $\{x_t\}_{t\geq 1}$ of a system $\Sigma$. In order to estimate $\mathrm{ind}_\varepsilon(\{x_t\})$ based on the sample $\{\tilde{x}_t\}_{t=1}^N$ we will derive a theoretical and computational method based on the covariance matrices defined in \eqref{cov_matrix}.

\subsubsection{$(0,1)$-matrices and eventual periodicity detection} Given $\varepsilon>0$ and some sampled-data $\{\tilde{x}_t\}_{t=1}^N$ from the orbit $\{x_t\}_{t\geq 1}\subset \mathbb{C}^n$ of a system $\Sigma$,  with $n\geq 2$. Let us consider the history data matrix $X\in \mathbb{C}^{n\times N}$ determined by the expression.
\begin{equation}
X=
\begin{bmatrix}
| & | &  & |\\
x_1 & x_2 & \cdots & x_N\\
| & |&   & |
\end{bmatrix}
\label{data_matrix_index}
\end{equation}

We will derive a data-driven structured perturbation result for $\mathrm{cov}(X)$.

\begin{lem}
\label{cov_pert}
Given two columns from $X\in\mathbb{C}^{n\times N}$ in \eqref{data_matrix_index} with $n\geq 2$, let us set $\varepsilon=\|x_j-x_k\|$. If we denote by $\mathrm{cov}_{j,k}(X)$ the $jk$ entry of $\mathrm{cov}(X)$, we will have that $\mathrm{cov}_{j,k}(X)$ satisfies the following constraint.
\begin{equation}
\left|\mathrm{cov}_{k,j}(X)-\hat{\sigma}(x_j)^2\right|\leq \frac{2\hat{\sigma}(x_j)}{\sqrt{n-1}}\varepsilon
\label{cor_jk_pert}
\end{equation}
\end{lem}
\begin{proof}
By definition of $\mathrm{cov}(X)$ in \eqref{cov_matrix} we have that.
\begin{equation*}
\mathrm{cov}_{j,j}(X)=\frac{\|x_j-\mu(x_j)\mathbf{e}_n\|^2}{n-1}=\hat{\sigma}(x_j)^2
\end{equation*}
Let us set $p=x_k-x_j$, then $x_k=x_j+p$ and this implies that.
\begin{align}
\mathrm{cov}_{k,j}(X)&=\frac{(x_k-\mu(x_k)\mathbf{e}_n)^\ast(x_j-\mu(x_j)\mathbf{e}_n)}{n-1}\nonumber\\
&=\frac{((x_j-\mu(x_j)\mathbf{e}_n)+(p-\mu(p)\mathbf{e}_n))^\ast(x_j-\mu(x_j)\mathbf{e}_n)}{n-1}\nonumber\\
&=\mathrm{cov}_{j,j}(X)+\frac{(p-\mu(p)\mathbf{e}_n)^\ast(x_j-\mu(x_j)\mathbf{e}_n)}{n-1}\nonumber\\
&=\hat{\sigma}(x_j)^2+\frac{(p-\mu(p)\mathbf{e}_n)^\ast(x_j-\mu(x_j)\mathbf{e}_n)}{n-1}
\label{cor_id_0}
\end{align}

Since $x_k=x_j+p$ and $x_j=x_k-p$ we will also have that,
\begin{align}
\hat{\sigma}(x_k)&=\frac{1}{\sqrt{n-1}}\|x_j+p-\mu(x_j+p)\mathbf{e}_n\|\nonumber\\
&\leq \frac{1}{\sqrt{n-1}}\|x_j-\mu(x_j)\mathbf{e}_n\|+\frac{1}{\sqrt{n-1}}\|p-\mu(p)\mathbf{e}_n\|\nonumber\\
&=\hat{\sigma}(x_j)+\hat{\sigma}(p)
\label{std_id_1}
\end{align}
and similarly.
\begin{align}
\hat{\sigma}(x_j)\leq\hat{\sigma}(x_k)+\hat{\sigma}(p)
\label{std_id_2}
\end{align}
By \eqref{std_id_1} and \eqref{std_id_2} we will have that.
\begin{align}
|\hat{\sigma}(x_j)-\hat{\sigma}(x_k)|&\leq \hat{\sigma}(p)\nonumber\\
&=\frac{1}{\sqrt{n-1}}\|p+\mu(p)\mathbf{e}_n\|\nonumber\\
&\leq \frac{1}{\sqrt{n-1}}(\|p\|+|\mu(p)|\|\mathbf{e}_n\|)\nonumber\\
&=\frac{1}{\sqrt{n-1}}(\|p\|+\frac{\sqrt{n}}{n}|\mathbf{e}_n^\ast p|)\nonumber\\
&\leq\frac{1}{\sqrt{n-1}}(\|p\|+\frac{\sqrt{n}}{n}\|\mathbf{e}_n\| \|p\|)\nonumber\\
&=\frac{2}{\sqrt{n-1}}\|p\|=\frac{2\varepsilon}{\sqrt{n-1}}
\label{cor_id_1}
\end{align}

By \eqref{cor_id_0} and \eqref{cor_id_1} we will have that.

\begin{align}
\left|\mathrm{cov}_{k,j}(X)-\hat{\sigma}(x_j)^2\right| &= \left|\frac{(p-\mu(p)\mathbf{e}_n)^\ast(x_j-\mu(x_j)\mathbf{e}_n)}{n-1}\right|\nonumber\\
&\leq  \frac{\|p-\mu(p)\mathbf{e}_n\|\|x_j-\mu(x_j)\mathbf{e}_n\|}{n-1}\nonumber\\
&= \hat{\sigma}(x_j)\hat{\sigma}(p)\leq \frac{2\hat{\sigma}(x_j)}{\sqrt{n-1}}\varepsilon
\end{align}
This completes the proof.
\end{proof}

Given $\delta>0$ and a data matrix $X$ in $\mathbb{C}^{n\times N}$ with $n\geq 2$, let us write $P_{D,\delta}(X)$ to denote the $(0,1)$-matrix $P_{D,\delta}(X)=[p_{kj}]$ in 
$\mathbb{C}^{N\times N}$ whose entries are defined by the expression.
\begin{equation}
p_{kj}=
\left\{
\begin{array}{l}
1, \: \: \mathrm{when} \:  \: |\mathrm{cov}_{k,j}(X)-\hat{\sigma}(x_j)^2|\leq \delta \: \mathrm{and} \: k\geq j\\
0, \: \: \mathrm{otherwise}
\end{array}
\right.
\label{Pdet_def}
\end{equation}

We will derive some eventual periodicity detection methods based on the matrix $P_{D,\delta}(X)$. In order to do this, let us start by considering the matrix $C_{k,n}\in \mathbb{C}^{n\times n}$ determined by the expression.
\begin{equation}
C_{k,n}=
\begin{bmatrix}
0 & 0 & 0 & \cdots & \delta_{k,1}\\
1 & 0 & 0 & \cdots & \delta_{k,2}\\
0 & 1 & 0 & \cdots & \delta_{k,3}\\
\vdots & \ddots & \ddots & \ddots & \vdots\\
0 & \cdots & 0 & 1 & \delta_{k,n}
\end{bmatrix}
\label{c_shift_first_rep}
\end{equation}
We call $C_{k,n}\in \mathbb{C}^{n\times n}$ a {\bf \em Generic Cyclic Shift} matrix or GCS in this document. It can be seen that a matrix $C_{k,n}\in \mathbb{C}^{n\times n}$ determined by \eqref{c_shift_first_rep} can be represented in the form.
\begin{equation}
C_{k,n}=\hat{e}_{k,n}\hat{e}_{n,n}^\ast+\sum_{j=1}^{n-1}\hat{e}_{j+1,n}\hat{e}_{j,n}^\ast
\label{c_shift_first_sum_rep}
\end{equation}

\begin{lem}
\label{shift_property}
The GCS matrix $C_{k,n}\in \mathbb{C}^{n\times n}$ satisfies the following conditions:
\begin{enumerate}
\item[(a)] $C_{k,n}\hat{e}_{j,n}=\hat{e}_{j+1,n}, \:\: 1\leq j\leq n-1$ 
\item[(b)] $C_{k,n}\hat{e}_{j,n}=\hat{e}_{k,n}, \:\: j=n$
\item[(c)] $C_{k,n}^n-C_{k,n}^{k-1}=\mathbf{0}_{n,n}$
\item[(d)] For any two integers $s\geq 0$ and $T\geq 1$ 
\begin{equation}
C_{s+1,s+T}^{s+mT}x=C_{s+1,s+T}^s x
\label{Sn_consntraints}
\end{equation}
for each integer $m\geq 0$ and each vector $x\in \mathbb{C}^{s+T}$. 
\item[(e)] If we define $P_{S,{s+1}}=\sum_{j=s+1}^{s+T}\hat{e}_{j,s+T}\hat{e}_{j,s+T}^\ast$ then,
\begin{align}
P_{C,s+1}^2&=P_{C,s+T}=P_{C,s+1}^\ast\\
P_{C,s+1}&=(P_{C,s+1}C_{s+1,s+T}P_{C,s+1})^{T}
\label{Sn_projective_constraints}
\end{align}
\end{enumerate}
\end{lem}
\begin{proof}
By \eqref{delta_def} and \eqref{c_shift_first_sum_rep} we will have that.
\begin{align*}
C_{k,n}\hat{e}_{j,n}&=\hat{e}_{k,n}(\hat{e}^\ast_{n,n}\hat{e}_{j,n})+\sum_{s=1}^{n-1}\hat{e}_{s+1,n}(\hat{e}_{s,n}^\ast\hat{e}_{j,n} )\\
&=\delta_{n,j}\hat{e}_{k,n}+\sum_{s=1}^{n-1}\delta_{s,j}\hat{e}_{s+1,n}\\
&=\left\{
\begin{array}{l}
\hat{e}_{j+1,n}, \:\: 1\leq j\leq n-1\\
\hat{e}_{k,n}, \:\: j=n
\end{array}
\right.
\end{align*}
It can be seen that the GCS matrix $C_{k,n}$ determined by \eqref{c_shift_first_rep} can be expressed in the form.
\begin{equation}
C_{k,n}=
\begin{bmatrix}
\hat{e}_{2,n} & \hat{e}_{3,n} & \hat{e}_{4,n} & \cdots & \hat{e}_{k,n}
\end{bmatrix}
\label{c_shift_second_rep}
\end{equation}

By \eqref{c_shift_second_rep} and elementary linear algebra we will have that $C_{k,n}$ is the companion matrix of the polynomial $p_k\in \mathbb{C}[z]$ determined by the expression. 
\begin{equation}
p_k(z)=z^n-z^{k-1}
\label{Ck_min_poly}
\end{equation}
This means that each GCS $C_{k,n}\in\mathbb{C}^{n\times n}$ satisfies the equation.
\begin{equation}
p_k(C_{k,n})=C_{k,n}^n-C_{k,n}^{k-1}=\mathbf{0}_{n,n}
\label{Ck_min_poly_cond}
\end{equation}
And this implies that.
 \begin{equation}
C_{k,n}^{n-(k-1)}C_{k,n}^{k-1}=C_{k,n}^n=C_{k,n}^{k-1}
\label{Ck_min_poly_cond_comp}
\end{equation}
By \eqref{Ck_min_poly_cond_comp} we will have that for any three integers $m,s\geq 0$ and $T\geq 1$, and for any $x\in \mathbb{C}^{s+T}$ we will have that.
\begin{equation}
C_{s+1,s+T}^{(s+mT)}x=C_{s+1,s+T}^{mT}C_{s+1,s+T}^{s}x=C_{s+1,s+T}^{s}x
\label{Ck_min_poly_cond_2}
\end{equation}
By definition of $P_{C,s+1}$ we will have that $P_{C,s+1}$ is a block diagonal matrix of the form.
\begin{equation}
P_{C,s+1}=
\begin{bmatrix}
\mathbf{0}_{s,s} & \mathbf{0}_{s,T}\\
\mathbf{0}_{T,s} & \mathbf{1}_{T}
\end{bmatrix}
\label{Pc_ident}
\end{equation} 
By \eqref{Pc_ident} it can be easily verified that $P_{C,s+1}^2=P_{C,s+T}$ and $P_{C,s+1}^\ast=P_{C,s+T}$. By \eqref{c_shift_second_rep} and \eqref{Pc_ident} we will have that.
\begin{equation}
P_{C,s+1}C_{s+1,s+T}P_{C,s+1}=
\begin{bmatrix}
\mathbf{0}_{s,s} & \mathbf{0}_{s,T}\\
\mathbf{0}_{T,s} & C_{1,T}
\end{bmatrix}
\label{Pc_block}
\end{equation} 
By (c) we will have that $C_{1,T}^T=\mathbf{1}_{T}$, and by \eqref{Pc_block} we will have that.
\begin{equation*}
(P_{C,s+1}C_{s+1,s+T}P_{C,s+1})^T=
\begin{bmatrix}
\mathbf{0}_{s,s} & \mathbf{0}_{s,T}\\
\mathbf{0}_{T,s} & C_{1,T}^T
\end{bmatrix}
=\begin{bmatrix}
\mathbf{0}_{s,s} & \mathbf{0}_{s,T}\\
\mathbf{0}_{T,s} & \mathbf{1}_{T}
\end{bmatrix}
=P_{C,s+1}
\label{Pc_block_2}
\end{equation*} 
This completes the proof.
\end{proof}

\begin{lem}
\label{exact_NEP}
Given two integers $s,T$ such that $s\geq 0$ and $T\geq 1$ and any matrix $A\in \mathcal{Z}_n(p)$ for $p(z)=z^{s+T}-z^s$, we will have that the system $(\Sigma,\mathcal{T})$ determined by the state-transition map $\mathcal{T}_t=A^t$ is eventually periodic and for any $x\in\mathbb{C}^n\backslash\{0\}$ we will have that $\mathrm{ind}_0(\{\mathcal{T}_t(x)\})=(s,T)$.
\end{lem}
\begin{proof}
Since $A\in \mathcal{Z}_n(p)$ we will have that $A^{s+T}-A^s=\mathbf{0}_{s+T,s+T}$, this implies that $A^{s+T}=A^s$ and that for any $x\in\mathbb{C}^n\backslash\{0\}$ and for each $t\in \mathbb{Z}^+$.
\begin{equation}
\mathcal{T}_{t+s+T}(x)=A^{t+s+T}x=A^t A^{s+T}x=A^t A^{s}x=A^{t+s}x=\mathcal{T}_{t+s}(x)
\label{exact_NEP_eq}
\end{equation}
By \eqref{exact_NEP_eq} we will have that if we set $\tilde{x}_t=A^tx$ for each $t\in \mathbb{Z}^+$, then $\tilde{x}_{t+s+T}=\tilde{x}_{t+s}$ and $\|\mathcal{T}_t(x)-\tilde{x}_t\|=0$, this implies that $\mathrm{ind}_{0}(\{\tilde{x}_t\})=(s,T)$. This completes the proof.
\end{proof}

Given $\varepsilon>0$ and some sampled data $\{\tilde{x}_t\}_{1\leq t\leq N}$ from an orbit $\{x_t\}_{t\geq 1}$ in the state space of a NEP system $\Sigma$. We will write $\mathrm{Ind}_\varepsilon(\{\tilde{x}_t\})=(s,T)$, if there exist two integers $s\geq 0$ and $T\geq 1$ such that $s+T\leq N-1$ and for each $k=1,\ldots , N-(s+T)$ and each $1\leq j\leq s+T$.
\begin{equation}
\left\|
\begin{bmatrix}
| &  & | \\
\tilde{x}_{k} & \cdots & \tilde{x}_{k+s+T-1}\\
| &  & |
\end{bmatrix} C_{s+1,s+T}\hat{e}_{j,s+T}-
\begin{bmatrix}
| \\
\tilde{x}_{k+j}\\
| 
\end{bmatrix}\right\|\leq \varepsilon
\label{sample_index}
\end{equation}
We will call the pair $(s,T)=\mathrm{Ind}_\varepsilon(\{\tilde{x}_t\})$ the sample index, the GCS matrix $C_{s+1,s+T}$ in \eqref{sample_index} will be called the GCS factor of the system $\Sigma$ based on the sample $\{\tilde{x}_t\}_{1\leq t\leq N}$, and we will say that the sample $\{\tilde{x}_t\}_{1\leq t\leq N}$ is meaningful if $\mathrm{Ind}_\varepsilon(\{\tilde{x}_t\})=\mathrm{ind}_{\varepsilon}(\{x_t\})$, $s+T\leq N-1$ for $(s,T)=\mathrm{Ind}_\varepsilon(\{\tilde{x}_t\})$, and $\|\tilde{x}_t-x_t\|\leq \varepsilon$ for each $1\leq t\leq s+T$.

\begin{lem}
\label{pattern_reading_lemma}
Given $\varepsilon>0$ and some sampled data $\{\tilde{x}_t\}_{1\leq t\leq N}$ from an orbit $\{x_t\}_{t\geq 1}$ in the state space of a NEP system $\Sigma\subseteq \mathbb{C}^n$ with $n\geq 2$. There is $\delta>0$ such that, if we set $X=[\tilde{x}_1~ \cdots ~\tilde{x}_N]$ and $P_{D,\delta}(X)_{k,j}$ denotes the $k,j$-entry of the matrix $P_{D,\delta}(X)$ defined in \eqref{Pdet_def}, then we will have that if $\mathrm{Ind}_\varepsilon(\{\tilde{x}_t\})=(s,T)$, then $P_{D,\delta}(X)_{q+mT,q}=1$ for each pair of integers $q$ and $m$ such that $q\geq s$, $m\geq 0$ and $q+mT\leq N$.
\end{lem}
\begin{proof}
Since $\mathrm{Ind}_\varepsilon(\{\tilde{x}_t\})=(s,T)$, by \eqref{sample_index} and by lemma \ref{shift_property} we will have that $\|\tilde{x}_{q+mT}-x_{q}\|\leq \varepsilon$ for each each pair of integers $q$ and $m$ such that $q\geq s$, $m\geq 0$ and $q+mT\leq N$. By lemma \ref{cov_pert} we will have that if we set $\delta=\frac{2\hat{\sigma}(x_j)}{\sqrt{n-1}}\varepsilon$, then $|\mathrm{cov}_{k,j}(X)-\hat{\sigma}(x_j)^2|\leq \delta$ by \eqref{cor_jk_pert} and this implies that $P_{D,\delta}(X)_{q+mT,q}=1$ by \eqref{Pdet_def}. This completes the proof.
\end{proof}

The previous lemma allows one to read the pattern left in a given sample by the daynamical laws of an NEP system under study. In section \S\ref{section_algorithms} we will present some prototypical algorithms based on lemma \ref{cov_pert} and 
lemma \ref{pattern_reading_lemma}.

\subsubsection{Controllers and algebraic matrix sets} 

Given $\varepsilon>0$, a polynomial $p\in\mathbb{C}[z]$, and a discrete-time dynamical system $(\Sigma,\mathcal{T})$ with $\Sigma\subset \mathbb{C}^n$, we will say that a family of orbits $\mathcal{O}(\Sigma)=\{\{y_t\}_{t\geq 1}\subset \Sigma\}$ is $\varepsilon$-almost controlled by the algebraic matrix set $\mathcal{Z}_m(p)$ for $1\leq m\leq n$, if for any orbit $\{x_t\}_{t\geq 1}$ in $\mathcal{O}(\Sigma)$ there is a sample $\{\tilde{x}_t\}_{1\leq t\leq N}\subseteq \{x_t\}_{t\geq 1}$ with $\tilde{x}_1=x_1$, a matrix $A\in \mathcal{Z}_N(p)$, and a vector $y\in \mathbb{C}^N$ such that for each integer $t\geq 1$ the following constraints are satisfied.
\begin{equation}
\left\|
\begin{bmatrix}
| &  & | \\
\tilde{x}_{1} & \cdots & \tilde{x}_{N}\\
| &  & |
\end{bmatrix} A^t y-
\mathcal{T}_{t}\left(
\begin{bmatrix}
| \\
x_{1}\\
| 
\end{bmatrix}\right)\right\|\leq \varepsilon
\label{almost_controller}
\end{equation}
The normed relations in \eqref{almost_controller} provide an alternative algebraic approach to the computation by dynamic mode decomposition of the connecting matrix representation $\mathbb{K}$ of the Koopman operator in the sense of \cite[\S2]{DMD_Schmid} and \cite{DMD_Kutz}, that is determined by some sampled-data $\{x_t\}_{t=1}^N$ in an orbit of some data-driven system under study, according to the equations $\mathbb{K}{x}_t={x}_{t+1}$, $1\leq t\leq N-1$.

In this study we will focus on the algebraic matrix sets that almost control the dynamical behavior of NEP systems.

\begin{thm}
\label{UAC_thm}
Given $\varepsilon>0$ and a meaningful sampled data $\{\tilde{x}_t\}_{1\leq t\leq N}$ in an orbit $\{x_t\}_{t\geq 1}$ of a NEP data-driven system $\Sigma$, if $\mathrm{Ind}_{\varepsilon}(\{\tilde{x}_t\})=(s,T)$, then $\{x_t\}_{t\geq 1}$ is $2\varepsilon$-almost controlled by $\mathcal{Z}_{s+T}(p)$ for $p(z)=z^{s+T}-z^s$. Furthermore, the sequence $\{\tilde{x}_t\}_{t\geq 1}$ determined by the recurrence relations
\begin{equation}
\begin{bmatrix}
| \\
\tilde{x}_{t+1}\\
| 
\end{bmatrix}=
\begin{bmatrix}
| &  & | \\
x_{1} & \cdots & x_{s+T}\\
| &  & |
\end{bmatrix} C_{s+1,s+T}^{t}\hat{e}_{1,s+T}
\label{recurrence_relations_thm_0}
\end{equation}
with $\tilde{x}_1=x_1$, satisfies the condition $\mathrm{Ind}_\varepsilon(\{\tilde{x}_t\})=(s,T)$, and for each $t\in\mathbb{Z}^+$ we have that $\|\tilde{x}_t-x_t\|\leq 2\varepsilon$.
\end{thm}
\begin{proof}
Since $\{\tilde{x}_t\}_{1\leq t\leq N}$ is meaningful and $\mathrm{Ind}_{\varepsilon}(\{\tilde{x}_t\})=(s,T)$, we will have that $\mathrm{ind}_\varepsilon(\{x_t\})=\mathrm{Ind}_\varepsilon(\{\tilde{x}_t\})$ and $\|\tilde{x}_t-x_t\|\leq \varepsilon$ for each $1\leq t\leq s+T$, since $\mathrm{ind}_\varepsilon(\{x_t\})=(s,T)$ this implies that there is an eventually periodic sequence $\{\hat{x}_t\}_{t\geq 1}$ with $\hat{x}_1\neq 0$, such that for each $t\in \mathbb{Z}^+$. 
\begin{equation}
\left\{
\begin{array}{l}
\|\hat{x}_t-x_t\|\leq \varepsilon\\
\hat{x}_{t+s}=\hat{x}_{t+s+T}
\end{array}
\right.
\label{NEP_cond}
\end{equation}
Let us consider the sequence $\{\tilde{x}_t\}_{t\geq 1}$ generated according to the following recurrence relations,
\begin{equation}
\begin{bmatrix}
| \\
\tilde{x}_{t+1}\\
| 
\end{bmatrix}=
\begin{bmatrix}
| &  & | \\
x_{1} & \cdots & x_{s+T}\\
| &  & |
\end{bmatrix} C_{s+1,s+T}^{t}\hat{e}_{1,s+T}
\label{recurrence_relation_thm}
\end{equation}
with $\tilde{x}_1=x_1$. By lemma \ref{shift_property} and by iterating on \eqref{sample_index} we will have that for each integer $t\geq 1$.
\begin{equation}
\left\|
\begin{bmatrix}
| &  & | \\
x_{1} & \cdots & x_{s+T}\\
| &  & |
\end{bmatrix} C_{s+1,s+T}^{t}\hat{e}_{1,s+T}-
\begin{bmatrix}
| \\
\tilde{x}_{t+1}\\
| 
\end{bmatrix}\right\|\leq \varepsilon
\label{sample_index_thm}
\end{equation}
By lemma \ref{exact_NEP} we will have that $\{\tilde{x}_t\}_{t\geq 1}$ is eventually periodic with $\mathrm{ind}_0(\{\tilde{x}_t\})=(s,T)=\mathrm{ind}_0(\{\hat{x}_t\})$, and by \eqref{NEP_cond} and \eqref{sample_index_thm} we will have that for each $t\in \mathbb{Z}^+$.
\begin{equation}
\|\hat{x}_t-\tilde{x}_t\|\leq \mathrm{max}_{1\leq t\leq s+T}\|\hat{x}_t-\tilde{x}_t\|=\mathrm{max}_{1\leq t\leq s+T}\|\hat{x}_t-x_t\|\leq \varepsilon
\label{NEP_approx_ident_1_thm}
\end{equation}
By \eqref{NEP_cond} and \eqref{NEP_approx_ident_1_thm} we will have that for each $t\in \mathbb{Z}^+$.
\begin{equation}
\|x_t-\tilde{x}_t\|\leq \|x_t-\hat{x}_t\|+\|\hat{x}_t-\tilde{x}_t\|\leq 2\varepsilon
\label{NEP_approx_ident_2_thm}
\end{equation}
By lemma \ref{shift_property} we also have that for $p(z)=z^{s+T}-z^s$, $p(C_{s+1,s+T})=\mathbf{0}_{s+T,s+T}$ and this implies that.
\begin{equation}
C_{s+1,s+T}\in \mathcal{Z}_{s+T}(p)
\label{mult_ident}
\end{equation}
By \eqref{sample_index_thm} and \eqref{mult_ident} we have that $\{x_t\}_{t\geq 1}$ is $2\varepsilon$-almost controlled by $ \mathcal{Z}_{s+T}(p)$ for $p(z)=z^{s+T}-z^s$. This completes the proof.
\end{proof}

\subsection{Cyclic realizations}
\label{CMR}
Given an orbit $\{x_t\}_{t\geq 1}$ of a nearly eventually periodic data-driven system $\Sigma$ determined by \eqref{data_driven_sys_def}, we will approach the computation of eventually periodic $\varepsilon$-approximate representations of the state-transition matrices $\{\mathcal{T}_{s}\}_{s\geq 1}$ that satisfy the equations $\mathcal{T}_{s}x_t=x_{t+s}$, by computing a discrete-time switched system $\hat{\Sigma}$ (in the sense of \cite[\S4.2]{bilinear_systems}) that is determined by the decomposition
\begin{equation}
\hat{\Sigma}:\left\{
\begin{array}{l}
\hat{x}_{t+1}=\hat{\mathcal{T}}_t \hat{x}_1\\
\hat{x}_{1}=x_1\\
\|x_{t}-\hat{x}_t\|\leq \varepsilon
\end{array}
\right.,t\geq 1
\label{CLS_Rep}
\end{equation}
related to some available sampled data $\{\tilde{x}_t\}_{1\leq t\leq N}\subseteq \Sigma$, with $\hat{\Sigma}\subseteq \{\tilde{x}_t\}_{1\leq t\leq N}$ and where each matrix $\hat{\mathcal{T}_t}$ has to be determined based on the sampled data in $\hat{\Sigma}$. 

Given some sampled data $\{\tilde{x}_t\}_{t=1}^N$ in an orbit of some data-driven system $\Sigma$ such that $\mathrm{dim}~ \mathrm{span}(\{\tilde{x}_1,\ldots,\tilde{x}_{N-1}\})=N-1$. Let us write $X^{(0)}_{N-1}$ and $X^{(1)}_{N-1}$ to denote the data matrices in $\mathbb{C}^{n\times (N-1)}$ determined by the expressions.
\begin{align}
X^{(0)}_{N-1}&=
\begin{bmatrix}
| &   & | \\
\tilde{x}_{1} & \cdots & \tilde{x}_{N-1} \\
| &   & |
\end{bmatrix}\nonumber\\
X^{(1)}_{N-1}&=
\begin{bmatrix}
| &   & | \\
\tilde{x}_{2} & \cdots & \tilde{x}_{N}\\
| &   & |
\end{bmatrix}
\label{data_matrices}
\end{align}
Given a matrix $A\in \mathbb{C}^{(N-1)\times (N-1)}$ that (approximately) solves the matrix equation.
\begin{equation}
X^{(1)}_{N-1}=X^{(0)}_{N-1}A
\label{data_matrix_eq}
\end{equation}
Let us consider the {\em reduced} singular value decomposition.
\begin{equation}
X^{(0)}_{N-1}=USV
\label{data_matrix_svd}
\end{equation}
with $U\in \mathbb{C}^{n\times {(N-1)}}$, $V\in \mathbb{C}^{(N-1)\times {(N-1)}}$ and $S\in\mathbb{R}^{(N-1)\times (N-1)}$. Given $\delta>0$, we will write $X^{(0)}_{N-1,\delta}$ to denote the perturbed data matrix corresponding to $X^{(0)}_{N-1}$, with reduced singular value decomposition
\begin{equation}
X^{(0)}_{N-1,\delta}=US_{\delta}V
\label{data_matrix_svd_approximant}
\end{equation}
where $S_\delta=[\hat{s}_{ij}]$ is the diagonal matrix determined by the expression.
\begin{equation}
\hat{s}_{jj}=
\left\{
\begin{array}{r}
s_{jj}\:,\:\: s_{jj}>0\\
\delta\:,\:\: s_{jj}=0
\end{array}
\right.
\label{S_delta_def}
\end{equation}

Given an orbit $\{x_t\}_{t\geq 1}$ in a NEP data-driven system $\Sigma\subseteq \mathbb{C}^n$ and some meaningful sample $\{\tilde{x}_t\}_{t=1}^N$ from $\{x_t\}_{t\geq 1}$ 
with $\mathrm{Ind}_{\varepsilon}(\{x_t\})=(s,T)$ and $s+T\leq N-1$. We will study the existence of cyclic realizations $\Phi:\mathbb{C}^{m\times m}\to \mathbb{C}^{n\times n}$ for the nearly eventually periodic orbit $\{x_t\}_{t\geq 1}$, that satisfy \eqref{universal_algebraic_constraints_3}. 

\begin{thm}
\label{CR_thm}
 Given $\varepsilon>0$, for any meaningful sampled data $\{\tilde{x}_t\}_{1\leq t\leq N}$ from an orbit $\{x_t\}_{t\geq 1}$ of a NEP data-driven system $\Sigma\subseteq \mathbb{C}^n$, there is $\delta >0$ such that if $\mathrm{Ind}_\delta(\{\tilde{x}_t\})=(s,T)$, $1\leq s+T\leq N-1$ and $N\leq n$, then there are two matrices $A\in\mathbb{C}^{n\times n}$ and $U\in \mathbb{C}^{n\times {(s+T)}}$ such that $U^\ast U=\mathbf{1}_{s+T}$, $A\in \mathcal{Z}_{n,\varepsilon}(p)$ and 
$U^\ast AU\in \mathcal{Z}_{s+T,\varepsilon}(p)$ for $p(z)=z^{s+T+1}-z^{s+1}$, and $\|x_{t+1}-A x_t\|\leq \varepsilon$ for each $t\in\mathbb{Z}^+$. 
\end{thm}
\begin{proof}
Given $\varepsilon>0$. Let us consider any meaningful sampled data $\{\tilde{x}_t\}_{1\leq t\leq N}$ from an orbit $\{x_t\}_{t\geq 1}$ of a NEP data-driven system $\Sigma\subseteq \mathbb{C}^n$ with $N\leq n$, choose $\vartheta>0$ such that $\mathrm{Ind}_\vartheta(\{\tilde{x}_t\})=(s,T)$ for some non-negative integers $s,T$ such that $1\leq s+T\leq N-1$. Let us consider the reduced singular value decomposition of the data matrix $X^{(0)}_{s+T}$ corresponding to the sub-sample $\{\tilde{x}_t\}_{t=1}^{s+T+1}$ that is determined by the expression.
\begin{equation}
X^{(0)}_{s+T}=
\begin{bmatrix}
| & & |\\
\tilde{x}_1 & \cdots & \tilde{x}_{s+T}\\
| & & |
\end{bmatrix}=USV
\label{data_matrix_thm}
\end{equation}
Let us consider the reduced singular value decomposition of the corresponding perturbed data matrix.
\begin{equation}
X^{(0)}_{s+T,\vartheta}\begin{bmatrix}
| & & |\\
\hat{x}_1 & \cdots & \hat{x}_{s+T}\\
| & & |
\end{bmatrix}
=US_{\vartheta}V
\label{data_matrix_svd_approx_thm}
\end{equation}
We will have that,
\begin{equation}
\mathrm{rank}(X^{(0)}_{s+T})\leq\mathrm{rank}(X^{(0)}_{s+T,\vartheta})=\mathrm{rank}(S_{\vartheta})=s+T
\label{rank_cond_thm}
\end{equation}
and also that.
\begin{equation}
\|X^{(0)}_{s+T}-X^{(0)}_{s+T,\vartheta}\|=\|U(S-S_{\vartheta})V\|=\|S-S_{\vartheta}\|\leq \vartheta
\label{data_matrix_pert_thm}
\end{equation}
By \eqref{rank_cond_thm} we will have that $S_{\vartheta}\in \mathbb{R}^{(s+T)\times (s+T)}$ is invertible, let us set.
\begin{equation}
A=US_\vartheta VC_{s+1,s+T}V^\ast S_\vartheta^{-1}U^\ast
\label{CMR_computation}
\end{equation}
Since $U^\ast U=\mathbf{1}_n$ we will have that for any $X,Y\in \mathbb{C}^{(s+T)\times(s+T)}$, the linear map defined by the operation $\phi(Y)=UYU^\ast$ satisfies the following condition. 
\begin{equation}
\phi(XY)=UXYU^\ast=UXU^\ast UYU^\ast=\phi(X)\phi(Y)
\label{mul_ident_phi}
\end{equation}
By \eqref{CMR_computation} we will have that $U^\ast AU$ is similar to $C_{s+1,s+T}$, this in turn implies that 
$U^\ast AU\in \mathcal{Z}_{s+T}(q)$ for $q(z)=z^{s+T}-z^{s}$, and this implies that $U^\ast AU\in \mathcal{Z}_{s+T}(p)$ for $p(z)=z^{s+T+1}-z^{s+1}$. By \eqref{mul_ident_phi} we will have that.
\begin{equation}
p(A)=p(\phi(U^\ast A U))=\phi(p(U^\ast A U))=\phi(\mathbf{0}_{s+T,s+T})=\mathbf{0}_{n,n}
\label{matrix_root_ident}
\end{equation}
By lemma \ref{shift_property}, \eqref{CMR_computation} and \eqref{S_delta_def} we will have that for each $1\leq t\leq s+T-1$.
\begin{align}
A\hat{x}_{t}&=AX^{(0)}_{s+T,\vartheta}\nonumber\hat{e}_{t,s+T}\\
&=US_\vartheta VC_{s+1,s+T}V^\ast S_\vartheta^{-1}U^\ast US_\vartheta V\hat{e}_{t,s+T} \nonumber\\
&=US_\vartheta VC_{s+1,s+T}\hat{e}_{t,s+T}
=X^{(0)}_{s+T,\vartheta}\hat{e}_{t+1,s+T}
=\hat{x}_{t+1}
\label{CMR_approx_indetity_1}
\end{align}
Let us consider the sequence $\{\hat{x}_t\}_{t\geq 1}$ generated by the recurrence raltions.
\begin{equation}
\left\{
\begin{array}{l}
\hat{x}_{t+1}=A\hat{x}_{t},\\
\hat{x}_1=x_1
\end{array}
\right.
\label{recurrence_relations_CMR_thm}
\end{equation}
Sice $\{\tilde{x}_t\}_{t=1}^{N-1}$ is meaningful and $\mathrm{Ind}_\vartheta(\{\tilde{x}_t\})=(s,T)$ we will have that $\|\tilde{x}_{s+T+1}-\tilde{x}_{s+2}\|\leq \vartheta$, this fact together with \eqref{CMR_approx_indetity_1} will imply that for each $1\leq t\leq s+T$.
\begin{align}
\|\tilde{x}_{t+1}-A\tilde{x}_{t}\|&\leq\|\tilde{x}_{t+1}-\hat{x}_{t+1}\|+\|\hat{x}_{t+1}-A\hat{x}_{t}\|+\|A\hat{x}_{t}-A\tilde{x}_t\|\nonumber\\
&\leq \vartheta + \vartheta +\|A\|\vartheta=(\|A\|+2)\vartheta
\label{approx_CMR_estimate}
\end{align}
By lemma \ref{exact_NEP} we will have that $\{\hat{x}_t\}_{t\geq 1}$ satisfies the condition $\mathrm{ind}_{\vartheta}(\{\hat{x}_t\})=(s,T)$. Since $\mathrm{ind}_{\vartheta}(\{x_t\})=(s,T)$ we will have that there is an eventually periodic sequence $\{\tilde{x}_t\}_{t\geq 1}$ such that for each $t\in\mathbb{Z}^+$
\begin{equation}
\left\{
\begin{array}{l}
\|x_t-\tilde{x}_{t}\|\leq \vartheta,\\
x_{t+s+T}=x_{t+s}
\end{array}
\right.
\label{EP_approximant_thm}
\end{equation}
By \eqref{approx_CMR_estimate} and \eqref{EP_approximant_thm} we will have that.
\begin{align}
\|x_{t+1}-Ax_{t}\|&\leq \|x_{t+1}-\tilde{x}_{t+1}\|+\|\tilde{x}_{t+1}-\hat{x}_{t+1}\|+\|\hat{x}_{t+1}-A\hat{x}_t\|\nonumber\\
&+\|A\hat{x}_t-A \tilde{x}_t\|+\|A \tilde{x}_t-A x_t\|\nonumber\\
&\leq \|x_{t+1}-\tilde{x}_{t+1}\|+(\|A\|+1)\max_{1\leq t\leq s+T}(\|\tilde{x}_{t}-x_{t}\|+\|x_{t}-\hat{x}_{t}\|) \nonumber\\
& +\|\hat{x}_{t+1}-A\hat{x}_t\| +\|A\|\|\tilde{x}_t-x_t\|\nonumber \\
&\leq\vartheta+2(\|A\|+1)\vartheta+\vartheta+\|A\|\vartheta=(3\|A\|+4)\vartheta
\end{align}
Let us set $\delta=\vartheta$ and $\varepsilon=(3\|A\|+4)\vartheta$. This completes the proof.
\end{proof}

Given $\varepsilon>0$ and some meaningful sample $\{\tilde{x}_t\}_{t=1}^{N}$ from an orbit $\{x_t\}_{t\geq 1}$ of a NEP system $\Sigma\subset \mathbb{C}^{n}$. The matrices $A,U$ determined by theorem \ref{CR_thm} will be called an approximate {\em cyclic matrix realization} for $\{x_t\}_{t\geq 1}$ based on $\{\tilde{x}_t\}_{t= 1}^{N}$. This relation will be represented by the expression $(A,U)=\mathbb{CMR}(\{\tilde{x}_t\},\varepsilon,\delta)$.

\begin{cor}
\label{solvability_result_1}
Given $\varepsilon>0$ and a meaningful sample $\{\tilde{x}_t\}_{t=1}^{N}$ from an orbit $\{x_t\}_{t\geq 1}$ of a NEP system $\Sigma\subset \mathbb{C}^{n}$. There is $\delta>0$ such that if $\mathrm{Ind}_{\delta}=(s,T)$ and $s+T\leq n$, then the problem $(A,U)=\mathbb{CMR}(\{\tilde{x}_t\},\varepsilon,\delta)$ is solvable.
\end{cor}
\begin{proof}
A direct application of theorem \ref{CR_thm}.
\end{proof}

\begin{lem}
\label{approximation_lem_1}
Given $\varepsilon>0$ and a polynomial $p\in \mathbb{C}[z]$, there is $\delta>0$ such that for any $X,Y\in \mathbb{C}^{n\times n}$ such that $\|X\|\leq 2$, $\|p(X)\|\leq \delta$ and $\|X-Y\|\leq \delta$, we have that $\|p(Y)\|\leq \varepsilon$. 
\end{lem}
\begin{proof}
Let us set.
\begin{equation}
\delta=\max\left\{\|p(X)\|,\|X-Y\|\right\}
\label{delta_def_thm}
\end{equation}
Since $p\in \mathbb{C}[z]$ we  will have that $p(z)=a_0+a_1 z+\cdots a_n z^n$ for some $a_0,\ldots,a_n\in\mathbb{C}$. By \eqref{delta_def_thm} and by \cite[Theorem 6.1.10]{HornBook} we will have that,
\begin{align}
\|p(Y)\|&=\|p(X+(Y-X))-p(X)+p(X)\|\nonumber\\
&\leq\|p(X+(Y-X))-p(X)\|+\|p(X)\|\nonumber\\
&\leq p'_{abs}(\|X\|+\|X-Y\|)\|X-Y\|+\|p(X)\|\nonumber\\
&\leq p'_{abs}(2+\delta)\delta+\delta\leq(1+p'_{abs}(2+\delta))\delta
\label{polynomial_constraint}
\end{align}
where $p'_{abs}(z)$ denotes the first derivative of $p_{abs}(z)=|a_0|+|a_1|z+\cdots+|a_n|z^n$. By \eqref{polynomial_constraint} we will have that if we set $\varepsilon=(1+p'_{abs}(1+2\delta))\delta$, then $\varepsilon>0$ and $\|p(Y)\|\leq \varepsilon$. This completes the proof.
\end{proof}

\begin{thm}
\label{main_classification_theorem}
Given $\varepsilon>0$, we will have that there is $\delta >0$ such that for any meaningful sampled data $\{\tilde{x}_t\}_{1\leq t\leq N}$ from an orbit $\{x_t\}_{t\geq 1}$ of a NEP data-driven system $\Sigma\subseteq \mathbb{C}^n$ with $N\leq n$, $\mathrm{Ind}_\delta(\{\tilde{x}_t\})=(s,T)$ and $1\leq s+T\leq N-1$, and for any two matrices $A,P\in \mathbb{C}^{n\times n}$ with $\|A\|\leq 2$ and $P=WW^\ast$ for some $W \in\mathbb{C}^{n\times m}$ with $W^\ast W=\mathbf{1}_m$ and $m\leq s+T$, if $\|A\tilde{x}_t-\tilde{x}_{t+1}\|\leq \delta$ and $\|P \tilde{x}_t-\tilde{x}_t\|\leq \delta$ for each $1\leq t\leq s+T$, then the map $\Phi:\mathbb{C}^{m\times m}\to \mathbb{C}^{n\times n}$ defined by $\Phi(Y)=WYW^\ast$ is multiplicative, and there is a matrix $A_{\eta}\in\mathbb{C}^{m\times m}$ such that $\|x_{t+1}-\Phi(A_\eta)x_t\|\leq \varepsilon$, for each $t\in\mathbb{Z}^+$. If in addition we have that $\|PAP-A\|\leq \delta$ and $\|A^{s+T+1}-A^{s+1}\|\leq \delta$, then $\|A_\eta^{s+T+1}-A_\eta^{s+1}\|\leq \varepsilon$.
\end{thm}
\begin{proof}
Given $\varepsilon>0$. Let us consider some meaningful sampled data $\{\tilde{x}_t\}_{1\leq t\leq N}$ from an orbit $\{x_t\}_{t\geq 1}$ of a NEP data-driven system $\Sigma\subseteq \mathbb{C}^n$ with $N\leq n$. Choose $\vartheta_T>0$ such that $\mathrm{Ind}_{\vartheta_T}(\{\tilde{x}_t\})=(s,T)$ for some $s,T$ such that $1\leq s+T\leq N-1\leq n$, and two matrices $A,P\in \mathbb{C}^{n\times n}$ with $P=WW^\ast$ for some $W \in\mathbb{C}^{n\times m}$ with $W^\ast W=\mathbf{1}_m$ and $m\leq T$. Let us define
\begin{equation}
A_\eta=W^\ast AW
\label{A_eta_def}
\end{equation}
and let us set.
\begin{equation}
\nu = 2\max \left\{\vartheta_T,\max_{1\leq t\leq s+T}\|Px_t-x_t\|,\max_{1\leq t\leq s+T}\|Ax_t-x_{t+1}\|\right\}
\label{nu_def_thm}
\end{equation}
By \eqref{A_eta_def} we will have that .
\begin{equation}
PAP=W A_\eta W^\ast
\label{pap_def_thm}
\end{equation}
Since $P=WW^\ast$ is clearly a projection by \eqref{nu_def_thm} and \eqref{pap_def_thm} we will have that for each $1\leq t\leq s+T$.
\begin{align}
\|WA_\eta W^\ast x_t-Ax_t\|&=\|PAPx_t-Ax_t\|\nonumber\\
&\leq \|PAPx_t-PAx_t\|+\|PAx_t-Px_{t+1}\|\nonumber\\
&+\|Px_{t+1}-x_{t+1}\|+\|x_{t+1}-Ax_{t}\|\nonumber\\
&\leq \|P\|\|A\|\|Px_t-x_t\|+\|P\|\|Ax_t-x_{t+1}\|\nonumber\\
&+\|Px_{t+1}-x_{t+1}\|+\|x_{t+1}-Ax_{t}\|\nonumber\\
&\leq(\|A\|+3)\nu
\label{main_normed_constraint}
\end{align}

By theorem \ref{UAC_thm} and by \eqref{nu_def_thm} we will have that the sequence $\{\tilde{x}_t\}_{t\geq 1}$ determined by the recurrence relations
\begin{equation}
\begin{bmatrix}
| \\
\tilde{x}_{t+1}\\
| 
\end{bmatrix}=
\begin{bmatrix}
| &  & | \\
x_{1} & \cdots & x_{s+T}\\
| &  & |
\end{bmatrix} C_{s+1,s+T}^{t}\hat{e}_{1,s+T}
\label{recurrence_relations_thm_2}
\end{equation}
with $\tilde{x}_1=x_1$, satisfies the condition $\mathrm{Ind}_\nu(\{\tilde{x}_t\})=(s,T)$, and for each $t\in\mathbb{Z}^+$ we have that.
\begin{equation}
\|\tilde{x}_t-x_t\|\leq 2\nu
\label{EP_approximant}
\end{equation}

By \eqref{main_normed_constraint} and \eqref{EP_approximant} we will have that for each $t\in\mathbb{Z}^+$.

\begin{align}
\|x_{t+1}-\Phi(A_\eta)x_t\|&=\|x_{t+1}-WA_\eta W^\ast x_t\|\nonumber\\
&=\|x_{t+1}-PAP x_t\|\nonumber\\
&\leq \|x_{t+1}-\tilde{x}_{t+1}\|+\|\tilde{x}_{t+1}-PAP \tilde{x}_t\|\nonumber\\
&+\|PAP\tilde{x}_{t+1}-PAPx_t\|\nonumber\\
&\leq 2\nu+(\|A\|+3)\nu+2\|A\|\nu=(3\|A\|+5)\nu
\label{A_eta_estimate}
\end{align}

Since $W^\ast W=\mathbf{1}_m$, we will have that for any $X,Y\in \mathbb{C}^{m\times m}$.
\begin{equation}
\Phi(XY)=WXYW^\ast=WXW^\ast WYW^\ast=\Phi(X)\Phi(Y)
\label{Phi_mult_ident_thm}
\end{equation}
By Gram-Schmidt orthogonalization theorem we will have that there is $W_{p}\in \mathbb{C}^{n\times (n-m)}$ such that $W_p^\ast W=\mathbf{0}_{n-m,m}$ and $W_p^\ast W_{p}=\mathbf{1}_{n-m}$, this implies that the matrix $Z=[W~ W_p]$ is unitary and also that.
\begin{equation}
Z^\ast W=
\begin{bmatrix}
\mathbf{1}_{m}\\
\mathbf{0}_{n-m,m}
\end{bmatrix}
\label{direct_sum_identity}
\end{equation}
By \eqref{direct_sum_identity} and by unitary invariance of the spectral norm we will have that for any matrix $X\in \mathbb{C}^{m\times m}$.
\begin{equation}
\|WXW^\ast\|=\|Z^\ast WXW^\ast Z\|=
\left\|
\begin{bmatrix}
X & \mathbf{0}_{m,n-m}\\
\mathbf{0}_{n-m,m} & \mathbf{0}_{n-m,n-m}
\end{bmatrix}
\right\|
=\|X\|
\label{norm_preserving_identity_0}
\end{equation}
By \eqref{norm_preserving_identity_0} we will have that for any $X\in \mathbb{C}^{m\times m}$.
\begin{equation}
\|\Phi(X)\|=\|X\|
\label{norm_preserving_identity}
\end{equation}
Let us set. 
\begin{equation}
\varrho=\max\left\{\|PAP-A\|,\|A^{s+T+1}-A^{s+1}\|\right\}
\label{rho_def_thm}
\end{equation}
By \eqref{pap_def_thm} and \eqref{rho_def_thm} we will have that.
\begin{align}
\|A-WA_\eta W^\ast\|&= \|A-PAP\|\leq \varrho
\label{second_normed_constraint}
\end{align}
By lemma \ref{approximation_lem_1} and by \eqref{polynomial_constraint}, \eqref{pap_def_thm}, \eqref{Phi_mult_ident_thm}, \eqref{norm_preserving_identity} and \eqref{second_normed_constraint} we will have that,
\begin{align}
\|A_\eta^{s+T+1}-A_\eta^{s+1}\|&=\|\Phi(A_\eta^{s+T+1}-A_\eta^{s+1})\|=\|\Phi(A_\eta)^{s+T+1}-\Phi(A_\eta)^{s+1}\|\nonumber\\
&=\|(WA_\eta W^\ast)^{s+T+1}-(WA_\eta W^\ast)^{s+1}\|\nonumber\\
&=\|(PAP)^{s+T+1}-(PAP)^{s+1}\|\nonumber\\
&\leq (1+(s+T+1)(2+\varrho)^{s+T}+(s+1)(2+\varrho)^s) \varrho
\label{pert_ident_approx_thm}
\end{align}
Let us set.
\begin{align*}
\delta&=\max\{\vartheta,\nu,\rho\}\\
\varepsilon&=\max\{11\delta, (1+(s+T+1)(2+\delta)^{s+T}+(s+1)(2+\delta)^s) \delta\}
\end{align*}
This completes the proof.
\end{proof}

Given $\varepsilon>0$ and some meaningful sample $\{\tilde{x}_t\}_{t=1}^{N}$ from an orbit $\{x_t\}_{t\geq 1}$ of a NEP system $\Sigma\subset \mathbb{C}^{n}$ with $N\leq n$. Let us conider the pair $(W,A_\eta)$ determined by theorem \ref{main_classification_theorem} for the sample $\{\tilde{x}_t\}_{t=1}^{N}$ and any two matrices $A,P\in\mathbb{C}^{n\times n}$ that satisfy the conditions in the statement of theorem \ref{main_classification_theorem} for some $\delta>0$, that have been computed using the solution to problem $(A,U)=\mathbb{CMR}(\{\tilde{x}_t\},\varepsilon,\delta)$ determined by theorem \ref{CR_thm}. The pair $(W,A_\eta)$ will be called an approximate {\em cyclic reduced order model} for $\{x_t\}_{t\geq 1}$ with respect to $(\{\tilde{x}_t\}_{t= 1}^{N}$,$\varepsilon,\delta)$. This relation will be represented by the expression $(W,A_\eta)=\mathbb{CROM}(\{\tilde{x}_t\},\varepsilon,\delta)$.

\begin{thm}
\label{CROM_solvability}
Given $\varepsilon>0$ and a meaningful sample $\{\tilde{x}_t\}_{t=1}^{N}$ from an orbit $\{x_t\}_{t\geq 1}$ of a NEP system $\Sigma\subset \mathbb{C}^{n}$. There is $\delta>0$ such that if $\mathrm{Ind}_{\delta}=(s,T)$ and $s+T\leq n$, then the problem $(W,A_\eta)=\mathbb{CROM}(\{\tilde{x}_t\},\varepsilon,\delta)$ is solvable.
\end{thm}
\begin{proof}
Given $\varepsilon>0$, and a meaningful sample $\{\tilde{x}_t\}_{t=1}^{N}$ from an orbit $\{x_t\}_{t\geq 1}$ of a NEP system $\Sigma\subset \mathbb{C}^{n}$ with $N\leq n$. By lemma \ref{pattern_reading_lemma} and theorem \ref{main_classification_theorem} we have that there is $\delta > 0$ such that $\mathrm{Ind}_\delta(\{\tilde{x}_t\})=(s,T)$ with $1\leq s+T\leq N-1$. 

Let us set $m=s+T$, and let us consider the reduced singular value decomposition. 
\begin{equation}
USV=[x_1~\cdots~x_{m}]
\label{svd_crom}
\end{equation}
Let us compute the perturbation $S_\delta$ of $S$ in \eqref{svd_crom} according to \eqref{S_delta_def} and \eqref{data_matrix_svd_approx_thm}. Let us set. 
\begin{equation}
r=\min\{\max\{1\leq t\leq s+T~ | ~ s_{tt}\geq \delta\},s+T\}
\label{r_def}
\end{equation}

Let us choose the first $r$ columns of $U$ and set. 
\begin{equation}
W=
\begin{bmatrix}
| & & |\\
u_1 & \cdots & u_r\\
| & & |
\end{bmatrix}
\label{W_computation}
\end{equation}
Since $r\leq s+T$ by \eqref{r_def}, and since by theorem \ref{CR_thm} the problem $(A,U)=\mathbb{CMR}(\{\tilde{x}_t\},\varepsilon,\delta)$ is solvable. If we set, 
\begin{equation}
X_r=W^\ast
\begin{bmatrix}
| & & |\\
x_1 & \cdots & x_{s+T}\\
| & & |
\end{bmatrix}
=\begin{bmatrix}
| & & |\\
\hat{x}_1 & \cdots & \hat{x}_{s+T}\\
| & & |
\end{bmatrix}
\label{X_r_def}
\end{equation}
then there is an approximate low rank solution for the problem 
\begin{equation}
W^\ast AW\begin{bmatrix}
| & & |\\
\hat{x}_1 & \cdots & \hat{x}_{r}\\
| & & |
\end{bmatrix}
=
\begin{bmatrix}
| & & |\\
\hat{x}_2 & \cdots & \hat{x}_{r+1}\\
| & & |
\end{bmatrix}
=
\begin{bmatrix}
| & & |\\
\hat{x}_1 & \cdots & \hat{x}_{r}\\
| & & |
\end{bmatrix}\hat{C}
\label{A_eta_LSP}
\end{equation}
and in particular, the matrix equation
\begin{equation}
\begin{bmatrix}
| & & |\\
\hat{x}_1 & \cdots & \hat{x}_{r}\\
| & & |
\end{bmatrix}\hat{C}
=
\begin{bmatrix}
| & & |\\
\hat{x}_2 & \cdots & \hat{x}_{r+1}\\
| & & |
\end{bmatrix}
\label{A_eta_LSP}
\end{equation}
admits a least squares approximate solution $\hat{C}\in\mathbb{C}^{r\times r}$. Let us compute the singular value decomposition. 
\begin{equation}
U_rS_rV_r=
\begin{bmatrix}
| & & |\\
\hat{x}_1 & \cdots & \hat{x}_{r}\\
| & & |
\end{bmatrix}
\label{svd_crom_2}
\end{equation}
If we set $A_\eta=U_rS_rV_r \hat{C}V_r^\ast S_r^{-1}U_r^\ast$, then by \eqref{A_eta_LSP} we will have that 
the matrices $\hat{A}=WA_\eta W^\ast$ and $P=WW^\ast$ satisfy the conditions in the statement of theorem \ref{main_classification_theorem}, and by \ref{CR_thm} we will have that $A_\eta\in\mathcal{Z}_{r,\varepsilon}(p)$ for $p(z)=x^{s+T+1}-z^{s+1}$. This implies that $W,A_\eta$ solve the problem $(W,A_\eta)=\mathbb{CROM}(\{\tilde{x}_t\},\varepsilon,\delta)$. This completes the proof.
\end{proof}

Given $\varepsilon>0$ and a meaningful sample $\{\tilde{x}_t\}_{t=1}^{N}$ from an orbit $\{x_t\}_{t\geq 1}$ of a NEP system $\Sigma\subset \mathbb{C}^{n}$ with $N\leq n$. If the problem $(W,A_\eta)=\mathbb{CROM}(\{\tilde{x}_t\}$ $,$ $\varepsilon$ $,$ $\delta)$ is solvable, and in addition $\|A_\eta A_\eta^\ast-\mathbf{1}_r\|\leq \delta$ and $\|A_\eta^\ast A_\eta-\mathbf{1}_r\|\leq \delta$. We can consider the nearness problem determined by the computation of a unitary matrix $U_\eta\in\mathbb{C}^{r\times r}$ such that $\|U_\eta-A_\eta\|\leq \varepsilon$. The problem determined by the computation of such a unitary will be called a {\em unitary cyclic reduced order model} for $\{x_t\}$ with respect to $(\{\tilde{x}_t\},\varepsilon,\delta)$. This relation will be represented by the expression $(W,U_\eta)=\mathbb{UCROM}(\{\tilde{x}_t\},\varepsilon,\delta)$.

\begin{thm}
\label{UCROM_solvability}
Given $\varepsilon>0$ and some meaningful sample $\{\tilde{x}_t\}_{t=1}^{N}$ from an orbit $\{x_t\}_{t\geq 1}$ of a NEP system $\Sigma\subset \mathbb{C}^{n}$. There is $\delta>0$ such that if $\mathrm{Ind}_{\delta}=(s,T)$ and $s+T\leq n$, then the problem $(W,U_\eta)=\mathbb{UCROM}(\{\tilde{x}_t\},\varepsilon,\delta)$ is solvable, whenever $\|A_\eta A_\eta^\ast-\mathbf{1}_r\|\leq \delta$ and $\|A_\eta^\ast A_\eta-\mathbf{1}_r\|\leq \delta$.
\end{thm}
\begin{proof}
Given $\varepsilon>0$, and a meaningful sample $\{\tilde{x}_t\}_{t=1}^{N}$ from an orbit $\{x_t\}_{t\geq 1}$ of a NEP system $\Sigma\subset \mathbb{C}^{n}$ with $N\leq n$. By theorem \ref{CROM_solvability} we will have that there is $\delta'>0$ such that if $\mathrm{Ind}_{\delta'}=(s,T)$ and $s+T\leq n$, then the problem $(W,A_\eta)=\mathbb{CROM}(\{\tilde{x}_t\},\varepsilon,\delta')$ is solvable. Let us set.
\begin{equation}
\delta=\max\{\delta',1/2\}
\label{delta_ucrom_def}
\end{equation}
Let us consider the singular value decomposition.
\begin{equation}
A_\eta=U_\delta S_\delta V_\delta
\label{svd_ucrom}
\end{equation}
By \eqref{delta_ucrom_def} we will have that $A_\eta$ is invertible and that $U_\eta=A_\eta(A_\eta^\ast A_\eta)^{-1/2}=U_\delta V_\delta$ is the unitary factor of the polar decomposition of $A_\eta$. Let us consider the spectrum $\sigma((A_\eta^\ast A_\eta)^{1/2})=\{\lambda_{1,\eta},\ldots,\lambda_{r,\eta}\}$ with eigenvalues counted with multiplicity. We will have that.
\begin{equation}
\max_{1\leq j\leq r}\left|\lambda_{j,\eta}^2 -1\right|=\left\|A_\eta^\ast A_\eta - \mathbf{1}_r\right\|\leq \delta
\label{spectral_constraints}
\end{equation}

By \eqref{spectral_constraints} we will have that for each $\lambda_{j,\eta}\in \sigma((A_\eta^\ast A_\eta)^{1/2})$.
\begin{equation}
\sqrt{1-\delta}\leq \lambda_{j,\eta} \leq \sqrt{1+\delta}
\label{spectral_constraints_2}
\end{equation}

This implies that.
\begin{align}
\|A_\eta-U_\eta\|=&\|V_\delta^\ast S_\delta V_\delta-\mathbf{1}_r\|=\| (A_\eta^\ast A_\eta)^{1/2}-\mathbf{1}_r\|\nonumber\\
=&\max_{1\leq j\leq r}\left|\lambda_{j,\eta} -1\right|\leq \sqrt{1+\delta}-1\leq \delta
\label{polar_ucrom}
\end{align}
By \eqref{approximation_lem_1} we will have that $U_\eta\in \mathcal{Z}_{r,\varepsilon'}(p)$ for $p(z)=z^{s+T+1}-z^{s+1}$, with $\varepsilon'=(1+(s+T+1)(2+\delta)^{s+T}+(s+1)(2+\delta)^s) \delta$. Let us set $\varepsilon=\max\{\delta,\varepsilon,\varepsilon'\}$, we will have that $(W, U_\eta)$ solves the UCROM problem with respect to $(\{\tilde{x}_t\},\varepsilon,\delta)$. This completes the proof.
\end{proof}

\section{Computational Methods}

\subsection{Algorithms}
\label{section_algorithms}
Given an orbit $\{x_t\}_{t\geq 1}$ of a NEP data-driven system $\Sigma$ determined by \eqref{data_driven_sys_def}, without loss of generality we can decompose the computation of the cyclic and reduced order cyclic universal controllers in two prototypical algorithms outlined in algorithm \ref{alg:main_alg_1} and algorithm \ref{alg:main_alg_2}.

\begin{algorithm}[h!]
\caption{Data-driven $\mathfrak{U_{C}}$ computation}
\label{alg:main_alg_1}
\begin{algorithmic}
\STATE{{\bf Data:}\:\:\:{\sc Tolerance $\varepsilon>0$, Sampled data:} $\hat{\Sigma}=\{x_t\}_{1\leq t\leq N}\subset \Sigma$}
\STATE{{\bf Result:}\:\:\: {{\sc UAC:} $(p,A,\varphi)=\mathfrak{U_{C}}(\{x_t\}_{1\leq t\leq N},x_1,\varepsilon)$}}
\begin{enumerate}
\STATE{Estimate $(s,T)=\mathrm{ind}_\varepsilon(\{x_t\})$ based on $\hat{\Sigma}=\{x_t\}_{1\leq t\leq N}$ and set $p(z)=z^{s+T}-z^s$\;}
\STATE{Compute $A\in \mathcal{Z}_m(p)$ and a completely positive linear map $\varphi:\mathbb{C}^{m\times m}\to \mathbb{C}^{n\times n}$ that satisfy \eqref{universal_algebraic_constraints_2}\;}
\end{enumerate}
\RETURN $(p,A,\varphi)$
\end{algorithmic}
\end{algorithm}

\begin{algorithm}[h!]
\caption{Data-driven $\mathfrak{U_{RC}}$ computation}
\label{alg:main_alg_2}
\begin{algorithmic}
\STATE{{\bf Data:}\:\:\:{\sc Tolerance $\varepsilon>0$, Sampled data:} $\hat{\Sigma}=\{x_t\}_{1\leq t\leq N}\subset \Sigma$}
\STATE{{\bf Result:}\:\:\: {{\sc UAC:} $(p,A_\eta,\Phi)=\mathfrak{U_{RC}}(\{x_t\}_{1\leq t\leq N},x_1,\varepsilon)$}}
\begin{enumerate}
\STATE{Estimate $(s,T)=\mathrm{ind}_\varepsilon(\{x_t\})$ based on $\hat{\Sigma}=\{x_t\}_{1\leq t\leq N}$ and set $p(z)=z^{s+T}-z^s$\;}
\STATE{Compute $A_\eta\in \mathbb{C}^{r\times r}$ with $r\leq m$ and a completely positive linear multiplicative map $\Phi:\mathcal{Z}_m(p)\to \mathbb{C}^{n\times n}$ that satisfy \eqref{universal_algebraic_constraints_3}.\;}
\end{enumerate}
\RETURN $(p,A_\eta,\Phi)$
\end{algorithmic}
\end{algorithm}

We have that the matrix techniques implemented in the proofs of lemma \ref{pattern_reading_lemma} and theorem \ref{CR_thm}, can be used to derive a prototypical data-driven cyclic matrix realization algorithm that is described by algorithm \ref{alg:second_alg}.

\begin{algorithm}[h!]
\caption{Data-driven $\mathbb{CMR}$ algorithm}
\label{alg:second_alg}
\begin{algorithmic}
\STATE{{\bf Data:}\:\:\: $\varepsilon,\delta>0$, $\{x_t\}_{1\leq t\leq s+ T+1}\subseteq \tilde{\Sigma}$, determined in the first step of algorithm \ref{alg:main_alg_1} applying lemma \ref{pattern_reading_lemma} and theorem \ref{CR_thm}}.
\STATE{{\bf Result:}\:\:\: { $(A,U)=\mathbb{CMR}(\{\tilde{x}_t\},\varepsilon,\delta)$ for $\tilde{\Sigma}$}}
\begin{enumerate}
\STATE{Set $m=s+T$\;}
\STATE{Compute the SVD $USV=[x_1~\cdots~x_{m}]$\;}
\STATE{Compute the perturbation $S_\delta$ of $S$ according to \eqref{S_delta_def} and \eqref{data_matrix_svd_approx_thm}\;}
\STATE{Set $A=US_\delta VC_{s+1,s+T}V^\ast S_\delta^{-1}U^\ast$\;}
\end{enumerate}
\RETURN $(A,U)$
\end{algorithmic}
\end{algorithm}

The matrix techniques implemented in the proofs of lemma \ref{pattern_reading_lemma}, theorem \ref{main_classification_theorem} and theorem \ref{CROM_solvability}, can be used to derive a prototypical data-driven cyclic matrix realization algorithm that is described by algorithm \ref{alg:third_alg}.

\begin{algorithm}[h!]
\caption{Data-driven reduced cyclic realization algorithm}
\label{alg:third_alg}
\begin{algorithmic}
\STATE{{\bf Data:}\:\:\: $\varepsilon,\delta>0$, $\{x_t\}_{1\leq t\leq s+T+1}\subseteq \tilde{\Sigma}$, for $(\delta,s,T)$ determined in the first step of algorithm \ref{alg:main_alg_2} applying lemma \ref{pattern_reading_lemma} and theorem \ref{main_classification_theorem}}.
\STATE{{\bf Result:}\:\:\: {$(W,A_\eta)=\mathbb{CROM}(\{x_t\},\varepsilon,\delta)$ for $\tilde{\Sigma}$}}
\begin{enumerate}
\STATE{Set $m=s+T$\;}
\STATE{Compute the SVD $USV=[x_1~\cdots~x_{m}]$\;}
\STATE{Compute the perturbation $S_\delta$ of $S$ according to \eqref{S_delta_def} and \eqref{data_matrix_svd_approx_thm}\;}
\STATE{Set $r=\min\{\max\{1\leq t\leq s+T~ | ~ s_{tt}\geq \delta\},s+T\}$\;}
\STATE{Choose the first $r$ columns of $U$ and set $W=[u_1 ~ \cdots ~u_r]$\;}
\STATE{Set $X_r=W^\ast [x_1~ \cdots ~ x_{s+T}]=[\hat{x}_1~ \cdots ~ \hat{x}_{s+T}]$\;}
\STATE{Solve $[\hat{x}_1~ \cdots ~ \hat{x}_{r}]\hat{C}=[\hat{x}_2~ \cdots ~ \hat{x}_{r+1}]$\;}
\STATE{Compute the SVD $U_rS_rV_r=[\hat{x}_1~ \cdots ~ \hat{x}_{r}]$\;}
\STATE{Set $A_\eta=U_rS_rV_r \hat{C}V_r^\ast S_r^{-1}U_r^\ast$\;}
\end{enumerate}
\RETURN $(A_\eta,W)$
\end{algorithmic}
\end{algorithm}

The matrix techniques implemented in the proofs of lemma \ref{pattern_reading_lemma}, theorem \ref{main_classification_theorem} and theorem \ref{UCROM_solvability}, can be used to derive a prototypical data-driven cyclic matrix realization algorithm that is described by algorithm \ref{alg:fourth_alg}.

\begin{algorithm}[h!]
\caption{Data-driven unitary cyclic reduced order modelling algorithm}
\label{alg:fourth_alg}
\begin{algorithmic}
\STATE{{\bf Data:}\:\:\: $\varepsilon,\delta>0$, $\{x_t\}_{1\leq t\leq s+T+1}\subseteq \tilde{\Sigma}$, for $(\delta,s,T)$ determined in the first step of algorithm \ref{alg:main_alg_2} applying lemma \ref{pattern_reading_lemma} and theorem \ref{main_classification_theorem}}.
\STATE{{\bf Result:}\:\:\: {$(W,A_\eta)=\mathbb{UCROM}(\{x_t\},\varepsilon,\delta)$ for $\tilde{\Sigma}$}}
\begin{enumerate}
\STATE{Apply algorithm \ref{alg:third_alg} to $(W,U_\eta)=\mathbb{CROM}(\{x_t\},\varepsilon,\delta)$ for $\tilde{\Sigma}$\;}
\STATE{Compute the SVD $U_\delta S_\delta V_\delta=U_\eta$\;}
\STATE{Set $U_\eta=U_\delta V_\delta$\;}
\end{enumerate}
\RETURN $(A_\eta,W)$
\end{algorithmic}
\end{algorithm}

\pagebreak

\subsection{Numerical Experiments}
\label{section_experiments}

In this section we will present some numerical simulations computed using UAC technology. These experiments were performed with {\sc Matlab}  R2018b Update 5 (9.5.0.1178774) 64-bit (glnxa64) and FreeFEM 4.200001 64bits. The FreeFEM programas used to generate the noisy input data signals, and the MatLab functions written to compute the universal algebraic controllers for the corresponding dynamical models are available at \cite{CodeVides}.

\subsubsection{UAC for predictive numerical simulation of Lam\'e systems}
Let us start considering the Navier equation for a steel sheet metal that can be written in the form
\begin{equation}
\left\{
\begin{array}{l}
(\lambda+\mu)\nabla(\nabla\cdot \mathbf{u})+\mu\nabla^2\mathbf{u}+\rho_0\mathbf{b}=\rho_0 \partial_t^2 \mathbf{u}(\mathbf{x},t)\\
\mathbf{u}(\mathbf{x},t)=\hat{\mathbf{u}}(\mathbf{x},t), \mathbf{x}\in \partial \mathcal{M}\\
\mathbf{u}(\mathbf{x},0)=u_0(\mathbf{x})\\
\partial_t\mathbf{u}(\mathbf{x},t)=u_1(\mathbf{x})
\end{array}
\right.
\label{Navier_eq_def}
\end{equation}
where the mechanical coefficients $\lambda,\mu$ are defined in terms of the corresponding Young's module $E$ and Poisson ratio $\nu$, according to the rules.
\begin{equation}
\left\{
\begin{array}{l}
\lambda=\frac{\nu E}{(1+\nu)(1-2\nu)}\\
\mu=\frac{E}{2(1+\nu)}
\end{array}
\right.
\label{constantes_de_Lame}
\end{equation}

We can apply algorithm \ref{alg:main_alg_1} and \ref{alg:main_alg_2} to compute some UAC for modal dynamic analysis correponding to mechanical models of the form \eqref{Navier_eq_def} under suitable boundary and inital conditions on a planar material $\mathbf{\Omega}\subseteq \mathbb{R}^2$ corresponding to a sheet metal. 

In order to simulate the signal data corresponding to a mechanical model of the form \eqref{Navier_eq_def}. We start by solving the reduced wave equation \eqref{discrete_Navier_eq} corresponding to \eqref{Navier_eq_def}, using finite element methods implemented in FreeFEM 4.2.
\begin{equation}
\left\{
\begin{array}{l}
(\lambda+\mu)\nabla(\nabla\cdot \mathbf{U})+\mu\nabla^2\mathbf{U}=\rho_0 \omega^2 \mathbf{U}\\
\mathbf{U}(\mathbf{x})=\hat{\mathbf{U}}(\mathbf{x}), \mathbf{x}\in \partial \mathcal{M}
\end{array}
\right.
\label{discrete_Navier_eq}
\end{equation}
Then, we use the Helmholtz solvent $\mathbf{u}(\mathbf{x},t)=e^{(i\omega t)}\mathbf{U}(\mathbf{x})$ of \eqref{Navier_eq_def} determined by \eqref{discrete_Navier_eq}, to compute the history data $\{\mathbf{U}^{(k)}\}_{k=1}^N=\{[\mathbf{U}_x^{(k)},\mathbf{U}_y^{(k)}]^\top\}_{k=1}^N$ and we save it to some file in a format that can be imported from MatLab. Once the history the data file produced by FreeFEM is available we import the {\em "noisy"} data to MatLab. We then apply UAC algorithm \ref{alg:main_alg_1} and UAC algorithm \ref{alg:main_alg_2} implemented in {\sc MatLab}, in order to compute a predictive numerical simulations for the displacement vector's sampled data $\Sigma_{sheet}=\{\mathbf{U}^{(k)}\}_{k=1}^N$, that are determined by the three UAC decompositions obtained by applying algorithm \ref{alg:second_alg}, algorithm \ref{alg:third_alg} and algorithm \ref{alg:fourth_alg}, that have the form.
\begin{equation}
\hat{\Sigma}_{sheet}:\left\{
\begin{array}{l}
\hat{U}_{t+1}=\hat{\mathcal{T}}_t \hat{U}_1\\
\hat{U}_{1}=\mathbf{U}^{(1)}\\
\|\mathbf{U}^{(t)}-\hat{U}_t\|\leq \varepsilon
\end{array}
\right.,t\geq 1
\label{CLS_Rep_sheet_metal}
\end{equation}
The graphical outputs corresponding to the the predictive numerical simulation for $\mathrm{Re}(\hat{U}_{t})$ computed with the UAC algorithm based on the $\mathbb{UCROM}$ method, is presented in figure \ref{fig:sheet_behavior}.
\begin{figure}[!h]
\centering
\includegraphics[scale=.66]{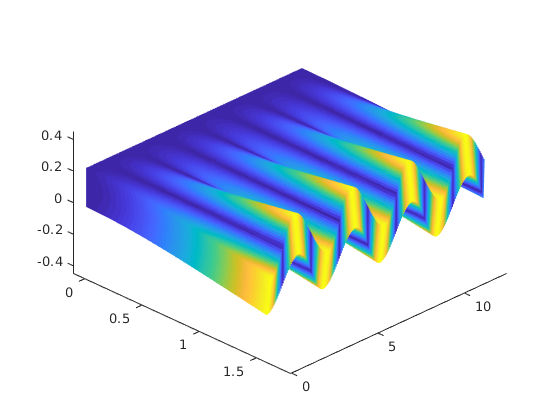}
\caption{Predictive numerical simulation computed using UAC with $Ind_\varepsilon(\Sigma_{sheet})=(0,197)$ for $\varepsilon=\mathcal{O}(1\times 10^{-3})$}
\label{fig:sheet_behavior}
\end{figure}
The pseudospectra of the connecting matrices $\hat{\mathcal{T}}_1$ in \eqref{CLS_Rep_sheet_metal} for each UAC method are presented in figure \ref{fig:sheet_pspectra}.
\begin{figure}[!h]
\centering
\includegraphics[scale=.33]{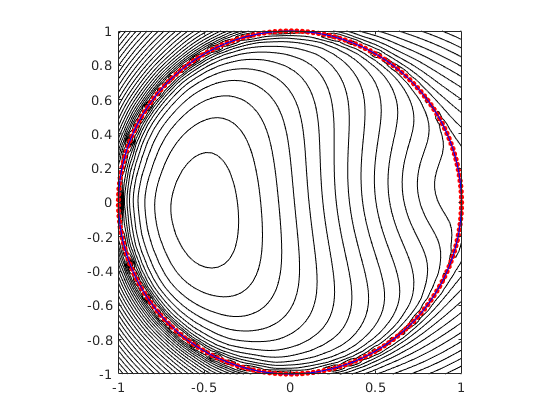}
\includegraphics[scale=.33]{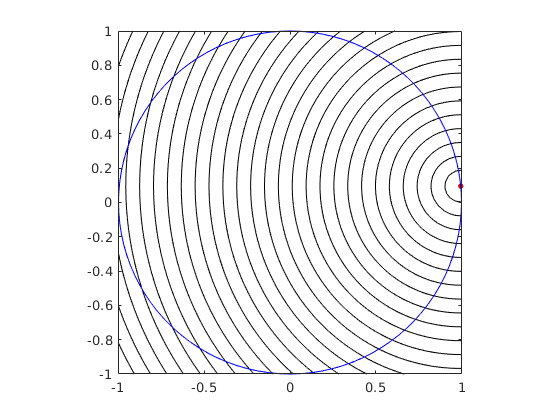}
\includegraphics[scale=.33]{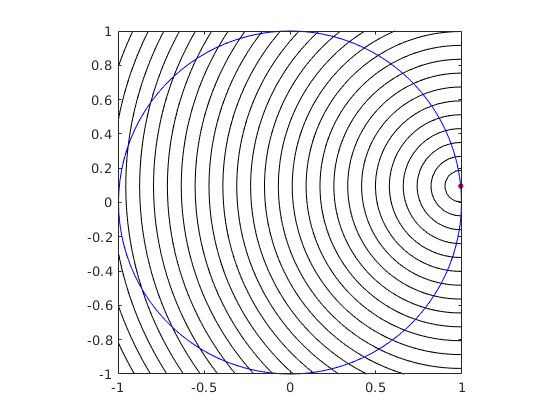}
\caption{Pseudospectra $\sigma_\varepsilon(\hat{\mathcal{T}}_1)$ of the low rank representations of the connecting matrices $\hat{\mathcal{T}}_1$: CMR method (left), CROM method (center) and UCROM method (right) with $Ind_\varepsilon(\Sigma_{sheet})=(0,197)$ for $\varepsilon=\mathcal{O}(1\times 10^{-3})$.}
\label{fig:sheet_pspectra}
\end{figure}
The relative errors with respect to $\|\cdot\|_\infty$ for each of the three methods are presented in \ref{fig:sheet_behavior_error}.
\begin{figure}[!h]
\centering
\includegraphics[scale=.56]{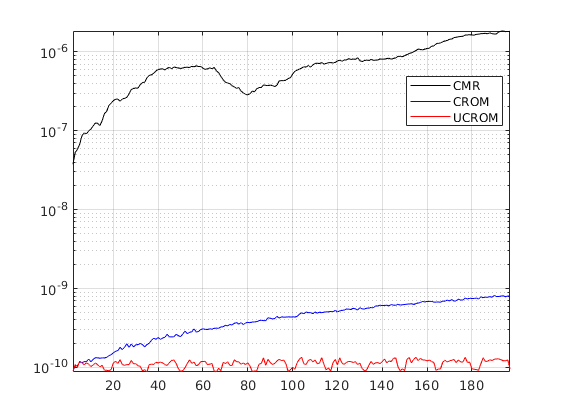}
\caption{Numerical simulation relative errors $\|\hat{U}_t-\mathbf{U}^{(t)}\|_\infty/\|\mathbf{U}^{(t)}\|_\infty$ with $Ind_\varepsilon(\Sigma_{sheet})=(0,197)$ for $\varepsilon=\mathcal{O}(1\times 10^{-3})$.}
\label{fig:sheet_behavior_error}
\end{figure}

\subsubsection{UAC for predictive numerical simulation of Navier-Stokes systems}
Let us start considering the Navier-Stokes equation for a von K\'arm\'an vortex street model that can be written in the form
\begin{align}
\frac{\partial {u}}{\partial{t}} +u\cdot\nabla u-\nu \Delta u+\nabla p&=0\nonumber \\
 \nabla\cdot u&=0 
\label{Navier_Stokes_eq_def}
\end{align}
where $u=(u_1,u_2)$ is the velocity vector and $p$ the pressure. For some suitable boundary and initial conditions for the velocity $u$ on the boundary $\Gamma$ of a planar region $\Omega\subseteq \mathbb{R}^2$.

We can apply algorithm \ref{alg:main_alg_1} and algorithm \ref{alg:main_alg_2} to compute some UAC for the numerical predictive simulation corresponding to mechanical models of the form \eqref{Navier_Stokes_eq_def} under suitable boundary and inital conditions on a planar region $\mathbf{\Omega}\subseteq \mathbb{R}^2$. 

In order to simulate the signal data corresponding to a mechanical model of the form \eqref{Navier_Stokes_eq_def}. We start by solving the difference equations \eqref{discrete_Navier_Stokes_eq} corresponding to \eqref{Navier_Stokes_eq_def}, using finite element methods implemented in FreeFEM 4.2.
\begin{equation}
\begin{split}\begin{array}{cl}
    \frac{1}{\tau } (u^{(n+1)}-u^{(n)}\circ X^{(n)}) -\nu\Delta u^{(n+1)} + \nabla p^{(n+1)} &=0,\\
    \nabla\cdot u^{(n+1)} &= 0
\end{array}\end{split}
\label{discrete_Navier_Stokes_eq}
\end{equation}
Then, we approximate the vorticites $\omega^{(k)}=\partial_x(u_2^{(k)})-\partial_y(u_1^{(k)})$ using the solvents $u^{(n)}=(u_1^{(n)},u_2^{(n)})$ of \eqref{Navier_Stokes_eq_def} determined by \eqref{discrete_Navier_Stokes_eq}, and save the history data $\{\omega^{(k)}\}_{k=1}^N$ to a file in a format that can be imported from MatLab. Once the history data file produced by FreeFEM is available we import the {\em "noisy"} data to MatLab. We then apply UAC algorithm \ref{alg:main_alg_1} and UAC algorithm \ref{alg:main_alg_2} implemented in {\sc MatLab}, in order to compute predictive numerical simulations for the vorticities' sampled data $\Sigma_\omega=\{\omega^{(k)}\}_{k=1}^N$, that are determined by the two UAC decompositions obtained by applying algorithm \ref{alg:second_alg} and algorithm \ref{alg:third_alg}, and have the form.
\begin{equation}
\hat{\Sigma}_{\omega}:\left\{
\begin{array}{l}
\hat{\omega}_{t+1}=\hat{\mathcal{T}}_t \hat{\omega}_1\\
\hat{\omega}_{1}=\mathbf{\omega}^{(1)}\\
\|\mathbf{\omega}^{(t)}-\hat{\omega}_t\|\leq \varepsilon
\end{array}
\right.,t\geq 1
\label{CLS_Rep_vortex}
\end{equation}
The graphical outputs corresponding to the the predictive numerical simulation for $\hat{\omega}_{t+1}$ computed with the UAC algorithm based on the $\mathbb{CRM}$ method, are presented in figure \ref{fig:vortex_behavior}.
\begin{figure}[!h]
\centering
\includegraphics[scale=.66]{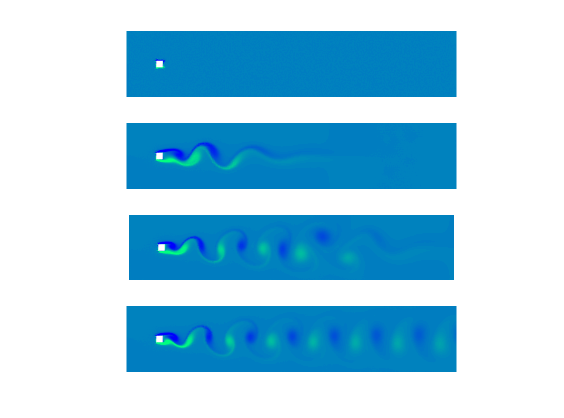}
\caption{Predictive numerical simulation computed using UAC with $Ind_\varepsilon(\Sigma_{\omega})=(140,17)$ for $\varepsilon=\mathcal{O}(1\times 10^{-7})$}
\label{fig:vortex_behavior}
\end{figure}
The pseudospectra of the connecting matrices $\hat{\mathcal{T}_1}$ in \eqref{CLS_Rep_vortex} for each UAC method are presented in figure \ref{fig:vortex_pspectra}.
\begin{figure}[!h]
\centering
\includegraphics[scale=.36]{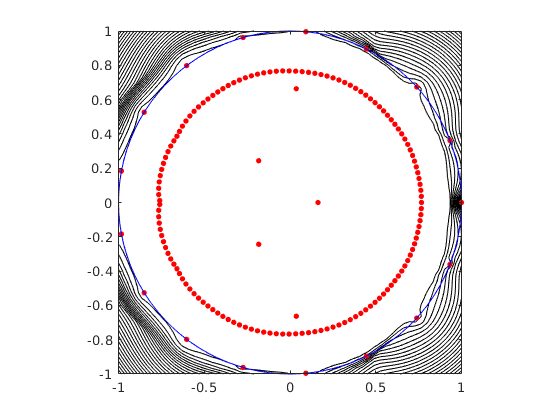}
\includegraphics[scale=.36]{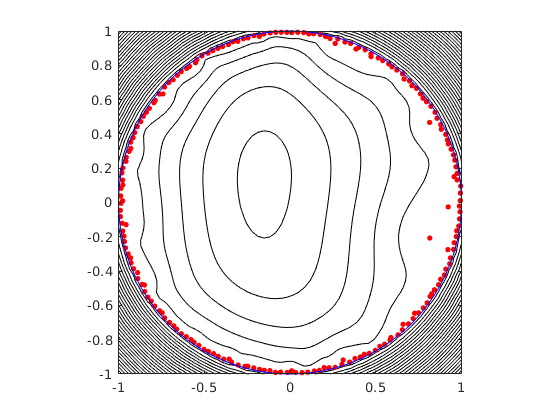}
\caption{Pseudospectra $\sigma_\varepsilon(\hat{\mathcal{T}}_1)$ of the low rank representations of the connecting matrices $\hat{\mathcal{T}}_1$: CMR method (left) and CROM method (right) with $Ind_\varepsilon(\Sigma_{\omega})=(140,17)$ for $\varepsilon=\mathcal{O}(1\times 10^{-7})$.}
\label{fig:vortex_pspectra}
\end{figure}
The relative errors with respect to $\|\cdot\|_\infty$ for the CMR and the CROM methods are presented in \ref{fig:sheet_behavior_error}.
\begin{figure}[!h]
\centering
\includegraphics[scale=.56]{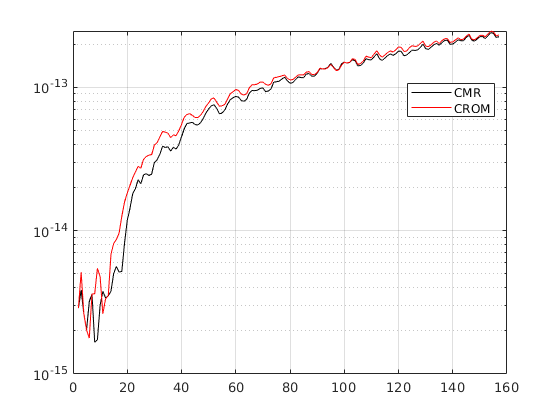}
\caption{Numerical simulation relative errors $\|\hat{\omega}_t-\omega^{(t)}\|_\infty/\|\omega^{(t)}\|_\infty$ with $Ind_\varepsilon(\Sigma_{\omega})=(140,17)$ for $\varepsilon=\mathcal{O}(1\times 10^{-7})$}
\label{fig:vortex_behavior_error}
\end{figure}

\subsubsection{UAC for predictive numerical simulation of time evolution of pure states in quantum systems}
Let us start by considering the Schr\"odinger equation for a quantum harmonic oscillator model that can be written in the form
\begin{equation}
i\frac{\partial \psi}{\partial t} -H\psi=0
\label{Schrodinger_eq_def}
\end{equation}
where the Hamiltonian $H$ is determined by the expression.
\begin{equation*}
H=-\frac{1}{2}\frac{\partial^2}{\partial x^2}+\frac{1}{2}x^2
\end{equation*}
For some suitable boundary and initial conditions for the wave function $\psi$ on the boundary $\partial\Omega$ of $\Omega=[-1,1]$. Let us consider 
a pure state $\psi(x,t)=e^{i\lambda t}\Psi(x)$. We will have that $\Psi$ satisfies the equation.
\begin{equation}
H\Psi=i\lambda \Psi
\label{Schrodinger_eq_def_comp}
\end{equation}
We can apply algorithm \ref{alg:main_alg_1} to compute a UAC for the numerical simulation of the time evolution of pure states that satisfy \eqref{Schrodinger_eq_def} and \eqref{Schrodinger_eq_def_comp}. 

In order to simulate the signal data corresponding to a quantum system determined by \eqref{Schrodinger_eq_def} and \eqref{Schrodinger_eq_def_comp}. We start by solving the difference equations \eqref{discrete_Schrodinger_eq} corresponding to \eqref{Schrodinger_eq_def} and \eqref{Schrodinger_eq_def_comp}, using finite difference methods implemented in {\sc MatLab}.
\begin{align}
\left(\mathbf{1}_N+\frac{i h_t}{2} H_h\right) \psi^{(n+1)}&=\left(\mathbf{1}_N-\frac{ih_t}{2}H_h\right)\psi^{(n)},\nonumber\\
    H_h\psi^{(1)} &= i\lambda\psi^{(1)}
\label{discrete_Schrodinger_eq}
\end{align}
Then, we computed the approximate solvents $\psi^{(n)}$ of \eqref{Schrodinger_eq_def} determined by \eqref{discrete_Schrodinger_eq}, that can be used as {\em "noisy"} input data for the UAC algorithm \ref{alg:main_alg_2} implemented in {\sc MatLab} as well, in order to compute predictive numerical simulations for the wave functions' sampled data $\Sigma_\psi=\{\psi^{(k)}\}_{k=1}^N$, that are determined by the two UAC decompositions obtained by applying algorithm \ref{alg:third_alg} and algorithm \ref{alg:fourth_alg}, that have the form.
\begin{equation}
\hat{\Sigma}_{\psi}:\left\{
\begin{array}{l}
\hat{\psi}_{t+1}=\hat{\mathcal{T}}_t \hat{\psi}_1\\
\hat{\psi}_{1}=\psi^{(1)}\\
\|{\psi}^{(t)}-\hat{\psi}_t\|\leq \varepsilon
\end{array}
\right.,t\geq 1
\label{CLS_Rep_QWave}
\end{equation}
The graphical outputs corresponding to the predictive numerical simulation for $\hat{\psi}_{t+1})$ computed with the UAC algorithm based on the $\mathbb{UCROM}$ method, are presented in figure \ref{fig:qwave_behavior}.
\begin{figure}[!h]
\centering
\includegraphics[scale=.54]{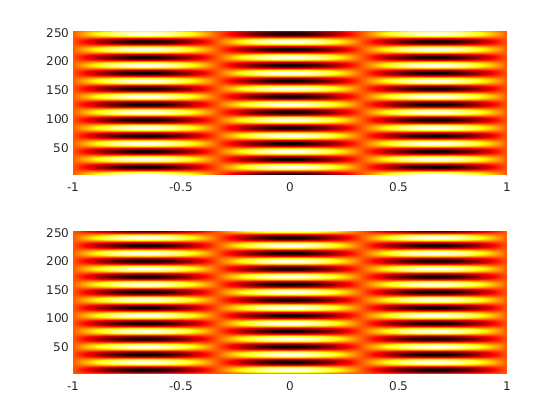}
\caption{Predictive numerical simulation computed using UAC for $\mathrm{Re}(\psi^{(t)})$ (top) and $\mathrm{Im}(\psi^{(t)})$ (bottom), with $Ind_\varepsilon(\Sigma_{\psi})=(0,136)$ for $\varepsilon=\mathcal{O}(1\times 10^{-3})$}
\label{fig:qwave_behavior}
\end{figure}
The pseudospectra of the connecting matrices $\hat{\mathcal{T}_1}$ in \eqref{CLS_Rep_QWave} for CROM and UCROM methods are presented in figure \ref{fig:qwave_pspectra}.
\begin{figure}[!h]
\centering
\includegraphics[scale=.36]{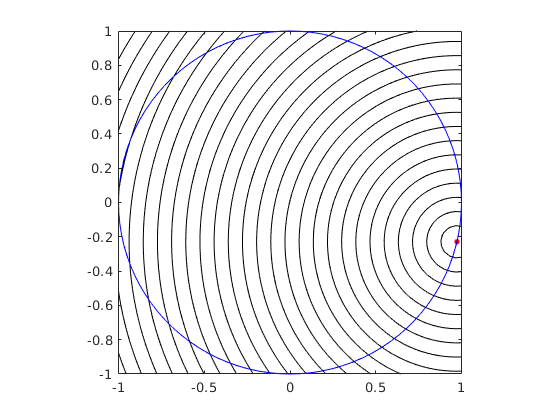}
\includegraphics[scale=.36]{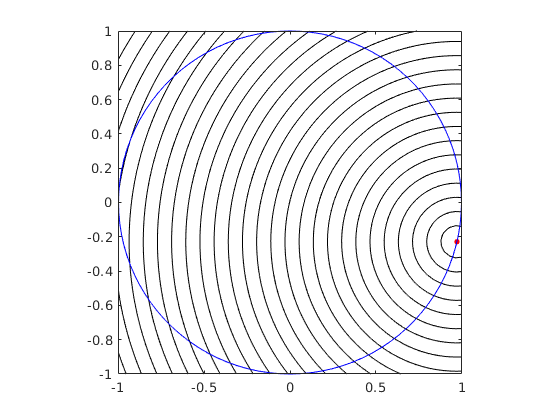}
\caption{Pseudospectra $\sigma_\varepsilon(\hat{\mathcal{T}}_1)$ of the low rank representations of the connecting matrices $\hat{\mathcal{T}}_1$: CROM method (left) and UCROM method (right) with $Ind_\varepsilon(\Sigma_{\psi})=(0,136)$ for $\varepsilon=\mathcal{O}(1\times 10^{-3})$.}
\label{fig:qwave_pspectra}
\end{figure}
The relative errors with respect to $\|\cdot\|_\infty$ for the CROM and the UCROM method are presented in \ref{fig:qwave_behavior_error}.
\begin{figure}[!h]
\centering
\includegraphics[scale=.56]{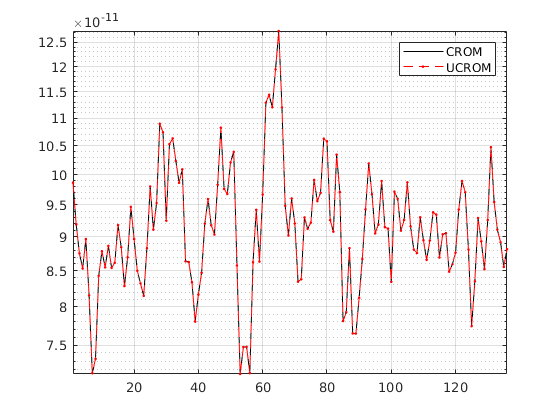}
\caption{Numerical simulation relative errors $\|\hat{\psi}_t-\psi^{(t)}\|_\infty/\|\psi^{(t)}\|_\infty$ with $Ind_\varepsilon(\Sigma_{\psi})=(0,136)$ for $\varepsilon=\mathcal{O}(1\times 10^{-3})$}
\label{fig:qwave_behavior_error}
\end{figure}

\section{Conclusion and Future Directions}
The results in \S\ref{section_UAC} allow one to derive computational methods like the ones described in \S\ref{section_algorithms}, for finite state approximation and predictive simulation of the dynamical behavior of a data-driven system determined by some data sampled from a set of valid/feasible states. 

Some applications of the algorithms in \S\ref{section_algorithms} to data-based schemes that can be used for model predictive control of industrial processes, will be presented in future communications.

The connections of the results in \S\ref{section_UAC} to the solution of problems related to controllability and realizability of finite-state systems in classical and quantum information and automata theory, in the sense of  \cite{finite_state_systems,finite_quantum_control_systems,finite_state_machine_approximation,BROCKETT20081}, will be further explored. 

Further applications of cyclic finite-state approximation schemes to industrial automation and Building Information Modeling technologies, will be the subject of future communications.

\section*{Acknowledgment}

The structure preserving matrix computations needed to implement the algorithms in \S\ref{section_algorithms}, were performed with {\sc Matlab}  R2018b Update 5 (9.5.0.1178774) 64-bit (glnxa64) and FreeFEM 4.200001 64bits at the Scientific Computing Innovation Center ({\bf CICC-UNAH}) of the National Autonomous University of Honduras. 

I would like to thank the Organizing Committees of GPOTS 2019 and COSy 2019 at Texas A\&M University and University of Regina, respectively, for the support received. Much of the research reported in this document was carried out while I was participating in the events.

I am grateful with Terry Loring, Marc Rieffel, Kenneth Davidson, Masoud Khalkhali, Jorge Destephen, Alexandru Chirvasitu, Douglas Farenick, Leonel Obando, Mario Molina, William F\'unez, Aner Godoy and Norman Sabill\'on for several interesting questions and comments, that have been very helpful for the preparation of this document.

\bibliographystyle{plain}
\bibliography{UniversalControllersFredyVides.bib}

\begin{thebibliography}{10}

\bibitem{finite_quantum_control_systems}
A.~M. {Bloch}, R.~W. {Brockett}, and C.~{Rangan}.
\newblock Finite controllability of infinite-dimensional quantum systems.
\newblock {\em IEEE Transactions on Automatic Control}, 55(8):1797--1805, Aug
  2010.

\bibitem{finite_state_systems}
R.~{Brockett} and A.~{Willsky}.
\newblock Finite group homomorphic sequential system.
\newblock {\em IEEE Transactions on Automatic Control}, 17(4):483--490, August
  1972.

\bibitem{BROCKETT20081}
R.~W. Brockett.
\newblock Reduced complexity control systems.
\newblock {\em IFAC Proceedings Volumes}, 41(2):1 -- 6, 2008.
\newblock 17th IFAC World Congress.

\bibitem{bilinear_systems}
D.~L. Elliott.
\newblock {\em Bilinear Control Systems Matrices in Action}, volume 169 of {\em
  Applied Mathematical Sciences}.
\newblock Springer, 2009.

\bibitem{FARHOOD2002417}
M.~Farhood and G.~E. Dullerud.
\newblock Lmi tools for eventually periodic systems.
\newblock {\em Systems \& Control Letters}, 47(5):417 -- 432, 2002.

\bibitem{HornBook}
R.~A. Horn.
\newblock {\em Topics in Matrix Analysis}.
\newblock Cambridge University Press, USA, 1986.

\bibitem{DMD_Kutz}
J.~L. Proctor, S.~L. Brunton, and J.~N. Kutz.
\newblock Dynamic mode decomposition with control.
\newblock {\em SIAM J Appl. Dyn. Syst.}, 15(1):142--161, 2016.

\bibitem{DMD_Schmid}
P.~J. Schmid.
\newblock Dynamic mode decomposition of numerical and experimental data.
\newblock {\em J. Fluid Mech.}, 656:5--28, 2010.

\bibitem{finite_state_machine_approximation}
D.~C. {Tarraf}.
\newblock An input-output construction of finite state $\rho/\mu$
  approximations for control design.
\newblock {\em IEEE Transactions on Automatic Control}, 59(12):3164--3177, Dec
  2014.

\bibitem{bookPspectra}
L.~Trefethen and M.~Embree.
\newblock {\em Spectra and Pseudospectra: The behavior of nonnormal matrices
  and operators}.
\newblock Princeton University Press, 01 2005.

\bibitem{Algebraic_sets}
F.~Vides.
\newblock On uniform connectivity of algebraic matrix sets.
\newblock {\em Banach J. Math. Anal.}, 13(4):918--943, 2019.

\bibitem{CodeVides}
F.~Vides.
\newblock Matlab functions for the computation of some universal algebraic
  controllers, 2020.
\newblock https://github.com/FredyVides/UAC.

\end{thebibliography}

\end{document}